\documentclass[a4paper, 11pt, reqno]{amsart}
\usepackage{amsfonts}
\usepackage{mathrsfs}
\usepackage{amsmath, amssymb, mathrsfs, amsthm, cases, verbatim}
\usepackage{pgflibraryarrows}
\usepackage{graphicx}
\usepackage[active]{srcltx}
\usepackage[colorlinks]{hyperref}
\usepackage{hypernat}

\usepackage{tikz}

\newtheorem{thm}{Theorem}[section]
\newtheorem{lem}[thm]{Lemma}
\newtheorem{prop}[thm]{Proposition}
\newtheorem{cor}[thm]{Corollary}

\theoremstyle{remark}
\newtheorem{rmk}{Remark}[section]

\numberwithin{equation}{section}

\theoremstyle{definition}
\newtheorem{defi}{Definition}[section]

\makeatletter
\newcommand{\rmnum}[1]{\romannumeral #1}
\newcommand{\Rmnum}[1]{\expandafter\@slowromancap\romannumeral #1@}
\makeatother

\newcommand{\abs}[1]{\lvert#1\rvert}

\newcommand{\dive}{\operatorname{div}}
\newcommand{\tr}{\operatorname{tr}}

\newcommand{\ol}{\overline}
\newcommand{\ul}{\underline}
\newcommand{\cd}{\cdot}
\newcommand{\ej}{{\varepsilon_j}}
\newcommand{\ppx}{{\partial\over \partial {p_1}}}
\newcommand{\ppy}{{\partial\over \partial {p_2}}}
\newcommand{\ppz}{{\partial\over \partial {p_3}}}
\newcommand{\pqx}{{\partial\over \partial {q_1}}}
\newcommand{\pqy}{{\partial\over \partial {q_2}}}
\newcommand{\pqz}{{\partial\over \partial {q_3}}}

\DeclareMathOperator*{\limsups}{limsup^\ast}
\DeclareMathOperator*{\liminfs}{liminf_\ast}

\def\vint_#1{\mathchoice%
          {\mathop{\kern 0.2em\vrule width 0.6em height 0.69678ex depth -0.58065ex
                  \kern -0.8em \intop}\nolimits_{\kern -0.4em#1}}%
          {\mathop{\kern 0.1em\vrule width 0.5em height 0.69678ex depth -0.60387ex
                  \kern -0.6em \intop}\nolimits_{#1}}%
          {\mathop{\kern 0.1em\vrule width 0.5em height 0.69678ex
              depth -0.60387ex
                  \kern -0.6em \intop}\nolimits_{#1}}%
          {\mathop{\kern 0.1em\vrule width 0.5em height 0.69678ex depth -0.60387ex
                  \kern -0.6em \intop}\nolimits_{#1}}}
\def\vintslides_#1{\mathchoice%
          {\mathop{\kern 0.1em\vrule width 0.5em height 0.697ex depth -0.581ex
                  \kern -0.6em \intop}\nolimits_{\kern -0.4em#1}}%
          {\mathop{\kern 0.1em\vrule width 0.3em height 0.697ex depth -0.604ex
                  \kern -0.4em \intop}\nolimits_{#1}}%
          {\mathop{\kern 0.1em\vrule width 0.3em height 0.697ex depth -0.604ex
                  \kern -0.4em \intop}\nolimits_{#1}}%
          {\mathop{\kern 0.1em\vrule width 0.3em height 0.697ex depth -0.604ex
                  \kern -0.4em \intop}\nolimits_{#1}}}

\newcommand{\kintint}[2]{\mathchoice%
          {\mathop{\kern 0.2em\vrule width 0.6em height 0.69678ex depth -0.58065ex
                  \kern -0.8em \intop}\nolimits_{\kern -0.45em#1}^{#2}}%
          {\mathop{\kern 0.1em\vrule width 0.5em height 0.69678ex depth -0.60387ex
                  \kern -0.6em \intop}\nolimits_{#1}^{#2}}%
          {\mathop{\kern 0.1em\vrule width 0.5em height 0.69678ex depth -0.60387ex
                  \kern -0.6em \intop}\nolimits_{#1}^{#2}}%
          {\mathop{\kern 0.1em\vrule width 0.5em height 0.69678ex depth -0.60387ex
                  \kern -0.6em \intop}\nolimits_{#1}^{#2}}}

\begin{document}
\title[Horizontal Mean Curvature Flow in the Heisenberg Group]{On the horizontal Mean Curvature Flow for Axisymmetric surfaces in the Heisenberg Group}
\author[F. Ferrari]{Fausto Ferrari}
\address{Dipartimento di Matematica dell'Universit\`a di Bologna, Piazza di Porta S. Donato, 5, 40126, Bologna, Italy}
\email{fausto.ferrari@unibo.it}
\author[Q. Liu]{Qing Liu}
\address{Department of Mathematics,
University of Pittsburgh,
Pittsburgh, PA 15260, USA}
\email{qingliu@pitt.edu}
\author[J. J. Manfredi]{Juan J. Manfredi}
\address{Department of Mathematics,
University of Pittsburgh,
Pittsburgh, PA 15260, USA}
\email{manfredi@pitt.edu}

\date{\today}

\thanks{F.~F.~is supported by  MURST, Italy, by University of Bologna,Italy
by EC project CG-DICE and by the ERC starting grant project 2011 EPSILON (Elliptic PDEs and Symmetry of Interfaces and Layers for Odd Nonlinearities). F.~F.~\ wishes to thank the Department of Mathematics at the University of Pittsburgh for the  kind hospitality.\\ \indent
Q.~L.~and J.~M.~are supported by NSF award DMS-1001179}



\maketitle

\begin{abstract}
We study the horizontal mean curvature flow in the Heisenberg group by using the level-set method. We prove the uniqueness, existence and stability of  axisymmetric viscosity solutions of the level-set equation. An explicit solution is given for the motion starting from a subelliptic sphere. We also give several properties of the level-set method and the mean curvature flow in the Heisenberg group. 
\end{abstract}

\section{Introduction}
We are interested in a family of compact hypersurfaces $\{\Gamma_t\}_{t\geq 0}$ in the \textit{Heisenberg group} parametrized by time $t\geq 0$. The motion of the hypersurfaces is governed by the following law:
\begin{equation}\label{original mcf}
V_H=\kappa_H,
\end{equation}
where $V_H$ denotes its \textit{horizontal normal velocity} and $\kappa_H$ stands for the \textit{horizontal mean curvature} in the Heisenberg group. The geometric motion (\ref{original mcf}) is thus called \textit{horizontal mean curvature flow}. The objective of this work is to investigate the evolution of the surface $\Gamma_t$ for  $t>0$ for a general class of initial surface $\Gamma_0$.

We implement a version of the level-set method adapted to the Heisenberg group. Let us assume, for the moment, that $\Gamma_t$ is smooth for any $t\geq 0$. If there exists $u\in C^2(\mathcal{H}\times [0, \infty))$ such that 
\[
\Gamma_t=\left\{p\in \mathcal{H}: u(p, t)=0\right \} 
\] 
for $t\ge 0$, 
then one may represent the horizontal normal velocity $V_H$ as 
\[
V_H={u_t\over |\nabla_H u|}
\]
and the horizontal mean curvature $\kappa_H$ as
\[
\kappa_H=\dive_H\left(\frac{\nabla_Hu}{|\nabla_H u|}\right)=\frac{1}{|\nabla_H u|}\tr\left[\bigg(I-\frac{\nabla_H u\otimes \nabla_H u}{|\nabla_H u|^2}\bigg)(\nabla_H^2 u)^\ast\right].
\]
Here $u_t$, $\nabla_H u$ and $(\nabla_H^2 u)^\ast$ respectively denote the derivative in $t$, the \textit{horizontal gradient} and the \textit{(symmetrized) horizontal Hessian} of $u$, and $\dive_H$ is the \textit{horizontal divergence operator}.  The horizontal gradient  of $u$ is given by $\nabla_H u=(X_1 u, X_2 u)$, where
\[
\begin{aligned}
& X_1=\frac{\partial}{\partial p_{1}}-\frac{p_{2}}{2}\frac{\partial}{\partial p_{3}};\\
& X_2=\frac{\partial}{\partial p_{2}}+\frac{p_{1}}{2}\frac{\partial}{\partial p_{3}}.\\
\end{aligned}
\]

In order to understand the law of motion by curvature (\ref{original mcf}), it therefore suffices to solve 
\begin{numcases}{(\textrm{MCF})\ }
u_t-\tr\left[\bigg(I-\frac{\nabla_H u\otimes \nabla_H u}{|\nabla_H u|^2}\bigg)(\nabla_H^2 u)^\ast\right]=0 &\text{in $\mathcal{H}\times (0, \infty)$,}\quad\label{mcf1} \\
u(p, 0)=u_0(p)  &\text{in $\mathcal{H}$.}\label{mcf2}
\end{numcases}
with a given function $u_0\in C(\mathcal{H})$ satisfying 
\[
\Gamma_0=\{p\in \mathcal{H}: u_0(p)=0\}.
\]
We refer the reader to \cite{CGG, ES, G1} for a detailed derivation of (MCF) in the Euclidean spaces and to \cite{CDPT, CC} for the analogue in the Heisenberg group. 

In this work, we aim to establish the uniqueness, existence and stability of the solutions of (MCF) that are spatially axisymmetric about the third coordinate axis. Namely, we are interested in the solutions $u$ satisfying 
\begin{equation}\label{axis}
u(p_1, p_2, p_3, t)=u(p_1', p_2', p_3, t) \text{ when $(p_1')^2+(p_2')^2=p_1^2+p_2^2$.}
\end{equation}
The symmetric structure of the functions is useful to obtain  positive results. We thus consider our contribution as a first step in order to prove more general results. Consult \cite{AAG, SS} for the results on motion by mean curvature for axisymmetric surfaces in the Euclidean spaces.

The symmetry with respect to the third axis in the Heisenberg group is not accidental. Indeed it is well known that this coordinate plays a key role in the Heisenberg group in several cases. In particular, we recall, for example, that $\{(0,0,p_3)\in\mathcal{H}:\:p_3\in \mathbb{R}\}$ is the center of the Heisenberg group and moreover the points along the $p_3$-axis correspond to conjugate points for the exponential map \cite{Montgomery}.  We warn the reader that, in general, our results do not apply to functions with different axes of symmetry. 

Hereafter the property (\ref{axis}) is sometimes referred to as ``spatial symmetry about the vertical axis'' or simply as ``axisymmetric''. 

Since the general regularity of $u$ is not known \textit{a priori},  we discuss the problem in the framework of viscosity solutions \cite{CIL}. As it is easily observed from the equation, a key difficulty lies at the \textit{characteristic set} of the level set $\Gamma_{t}$, i.e., at the points  where $\nabla_H u=0$. 

\subsection{Uniqueness} Even in the Euclidean case \cite{CGG, ES, S, G1}, the proof of the \textit{comparison principle} and the uniqueness of solutions for this type of degenerate equations need special techniques to deal with the characteristic set. The comparison principle we expect is as follows: for any \textit{upper semicontinuous subsolution} $u$ and \textit{lower semicontinuous supersolution} $v$ defined on $\mathcal{H}\times[0, \infty)$ satisfying $u(p, 0)\leq v(p, 0)$ for all $p\in \mathcal{H}$, we have $u(p, t)\leq v(p, t)$ for any $t\geq 0$.
L. Capogna and G. Citti \cite{CC} extended the results of \cite{ES} and proved a comparison principle by excluding the characteristic points. Their comparison principle further required that (\rmnum{1}) either $u$ or $v$ be uniformly continuous and (\rmnum{2}) the initial surfaces are completely separated in the horizontal directions, i.e., $u(p, 0)\leq v(q, 0)$ for all $p=(p_1, p_2, p_3), q=(q_1, q_2, q_3)\in \mathcal{H}$ such that $p_i=q_i$ for $i=1, 2$. The general comparison principle, as stated above, remains an open question. 

In this paper, we follow \cite{CGG, G1} and give a comparison principle without assuming those two conditions above but requiring that either $u$ or $v$ be axisymmetric. We also restrict ourselves to the case of compact surfaces for simplicity. The comparison theorem we present is as follows.
\begin{thm}[Comparison theorem]\label{comparison theorem}
Let $u$ and $v$ be respectively an upper semicontinuous subsolution and a lower semicontinuous supersolution of 
\[
u_t-\tr\left[\bigg(I-\frac{\nabla_H u\otimes \nabla_H u}{|\nabla_H u|^2}\bigg)(\nabla_H^2 u)^\ast\right]=0
\]
in $\mathcal{H}\times (0, T)$ for any $T>0$. Assume that there is a compact set $K\subset \mathcal{H}$ and $a, b\in \mathbb{R}$ with $a\leq b$ such that $u(p, t)=a$ and $v(p, t)=b$ for all $p\in \mathcal{H}\setminus K$ and $t\in [0, T]$. Assume that either $u$ or $v$ is spatially axisymmetric about the vertical axis. If $u(p, 0)\leq v(p, 0)$ for all $p\in \mathcal{H}$, then $u\leq v$ for all $(x, t)\in \mathcal{H}\times [0, T]$. 
\end{thm}

The uniqueness of the axisymmetric solutions follows immediately from the theorem above. It is worth remarking that when showing comparison principles involving viscosity solutions, one usually needs to double the variables and maximize 
\[
u(p, t)-v(q, s)-{\phi(p, q)+|t-s|^2\over \varepsilon},
\]
where $\varepsilon>0$, $p, q\in \mathcal{H}$, $t, s\in [0, \infty)$ and $\phi$ is a smooth penalty function on $\mathcal{H}\times \mathcal{H}$, and argues by contradiction.



The typical choice of $\phi$ in the Euclidean spaces, as discussed in \cite{CIL} and \cite{G1}, is a quadratic function $\phi(x, y)=|x-y|^2$ usually or a quartic function $\phi(x, y)=|x-y|^4$ for mean curvature flow equation (for $x, y\in \mathbb{R}^n$). The advantages of these choices are:
\begin{enumerate}
\item[(a)] The derivatives of $\phi$ with respect to $x$ and $y$ are opposite, i.e., $\nabla_x\phi=-\nabla_y\phi$. We would plug these derivatives in the viscosity inequalities, since they serve as semi-differentials for the compared functions. This construction enables us to derive a contradiction.
\item[(b)] When discussing (mild) singular equations such as curvature flow equations, it will be convenient to have the second derivatives be $0$ whenever the first derivatives are $0$,
as in the case of  $|x-y|^4$.
\end{enumerate}

The analogue of the choice $|x-y|^4$ is not immediate in the  Heisenberg group. Since the group multiplication is not commutative, the two natural options  $f(p, q)=|q^{-1}\cd p|^4$  and $g(p, q)=|p\cd q^{-1}|^4$ are different.  It seems that we have more options but it turns out that neither of them satisfies both conditions above. By direct calculation, we may find that $g$ fulfills the requirement (a) above but its derivatives do not satisfy (b). The function $f$ is good for our requirement (b) but unfortunately fails to have the property (a). Hence, the main difficulty of the uniqueness argument in the Heisenberg group consists in a wise choice of the penalty function $\phi$.

Our approach combines both choices $f$ and $g$. On one hand, we use $f$ to derive a relaxed definition (Definition \ref{trad defi}) of solutions of (\ref{mcf}), facilitating us to overcome the singularity.
On the other hand, under the help of axial symmetry, we obtain the property (b) when employing $g$ type of penalty functions in the proof of the comparison principle. The symmetry plays an important role since it largely simplifies the structure of characteristic points; see \cite{FLM2} for some geometric details.

Roughly speaking, when  a smooth function $u(p, t)$ is spatially symmetric about the vertical axis, i.e., $u=u(r, p_3, t)$, where $r=(p_1^2+p_2^2)^{1/2}$, we get
\[
\begin{aligned}
&X_1u={p_1\over r}\frac{\partial}{\partial r}u-{p_2\over 2}\frac{\partial}{\partial p_3}u;\\
&X_2u={p_2\over r}\frac{\partial}{\partial r}u+{p_1\over 2}\frac{\partial}{\partial p_3}u.
\end{aligned}
\]
Then $\nabla_H u(p, t)=0$ implies that either $\partial u/\partial r=\partial u/ \partial p_3=0$ or $p_1^2+p_2^2=0$. This observation enables us to obtain property (b) for a large power of the function $g$. 

Our definition of viscosity solutions is actually an extension of that introduced in \cite{CGG, G1} to the Heisenberg group. In Section \ref{sec defi}, we discuss the equivalence of this definition and the others.

\subsection{Existence} Generally speaking, there are at least three possible approaches to get the existence of solutions of (MCF). One may follow \cite{ES} to use the uniformly parabolic theory by considering a regularized equation
\begin{numcases}{}
u_t-\tr\left[\bigg(I-\frac{\nabla_H u\otimes \nabla_H u}{|\nabla_H u|^2+\varepsilon^2}\bigg)(\nabla_H^2 u)^\ast\right]=0 &\text{in $\mathcal{H}\times (0, \infty)$,}\nonumber \\
u(p, 0)=u_0(p)  &\text{in $\mathcal{H}$}\nonumber.
\end{numcases}
and take the limit of its solution as $\varepsilon\to 0$; see \cite{CC} for results in the \textit{Carnot groups} with this method.  Another possible option is to employ \textit{Perron's method} by considering the supremum of all subsolutions or the infimum of all supersolutions, as is shown in \cite{CGG, G1} for the Euclidean case. We refer to \cite{I, CIL} for a general introduction of this method in the framework of viscosity solutions. 

A third method for existence is based on the representation theorem involving optimal control or game theory, which recently generated a spur of activity. Consult the works \cite{CSTV, KS1, KS2, MPR1, MPR2, PSSW, PS, ST} for the development of this approach to various equations in  Euclidean spaces. For the mean curvature flow in the sub-Riemannian geometry, a stochastic control-based formulation  analogous to \cite{ST} is addressed in \cite{DDR}, where  the authors found a solution via a suitable optimal stochastic control problem.

In this work, we adapt the deterministic game-theoretic approach of R. V. Kohn and S. Serfaty \cite{KS1} to the Heisenberg group. For any given axisymmetric continuous function $u_0$, we set up a family of games, whose \textit{value functions} $u^\varepsilon$ converge to the solution $u$ to the mean curvature flow equation. We not only get the existence of solutions but also obtain a game interpretation of the equation in the Heisenberg group. The proof is based on the \textit{dynamic programming principle}, which can be regarded as  a (nonlinear) semigroup. Our convergence theorem relies on the comparison principle given in Theorem \ref{comparison theorem}. More precisely, taking the \textit{half relaxed limits}, defined on $\mathcal{H}\times [0, \infty)$,
\begin{equation}\label{upper limit}
  \begin{aligned}
  \overline{u}(p, t):&= {\limsups_{\varepsilon\to 0}} u^\varepsilon(p, t)\\
  &=\lim_{\delta\to 0}\sup\{u^\varepsilon(q, s): s\geq 0, 0<\varepsilon<\delta,\
  \abs{p-q}+\abs{t-s}<\delta\}\\
  \end{aligned}
\end{equation}
  and
\begin{equation}\label{lower limit}
  \begin{aligned}
  \underline{u}(p, t):&=\liminfs_{\varepsilon\to 0} u^\varepsilon(p, t)\\
  &=\lim_{\delta\to 0}\inf\{u^\varepsilon(q, s): s\geq 0, 0<\varepsilon<\delta,\
  \abs{p-q}+\abs{t-s}<\delta\},
  \end{aligned}
\end{equation}
we show that $\ol{u}$ and $\ul{u}$ are respectively a subsolution and a supersolution of (\ref{mcf1}) using the dynamic programming principle. We also show that $\ol{u}(p, 0)\leq u_0(p)\leq \ul{u}(p, 0)$ and that $u^\varepsilon, \ol{u}$ and $\ul{u}$ are spatially axisymmetric about the vertical axis. Our game approximation then follows immediately from the comparison theorem. See Section \ref{sec existence} for more details on the game setting and the existence theorem. 

We discuss asymptotic mean value properties related to random \textit{tug-of-war games} for $p$-harmonic functions on the Heisenberg group in \cite{FLM}.

\subsection{Stability and uniqueness of the evolution}
We give a stability theorem, which is used to show that the equation (\ref{mcf1}) is invariant under the change of dependent variable. We prove that for any continuous function $\theta: \mathbb{R}\to \mathbb{R}$, the composition $\theta\circ u$ is a solution provided that $u$ is a solution. Note that this is clear if $\theta$ is smooth and strictly monotone, since the mean curvature flow equation is \textit{geometric} and \textit{orientation-free}; see \cite{G1} for more explanation. Our stability result is applied so as to weaken the regularity of $\theta$.

It follows from the invariance property that any axisymmetric evolution $\Gamma_t$ does not depend on the particular choice of $u_0$ but depends on $\Gamma_0$ only, which is important for the level-set method.

\subsection{Evolution of spheres}
Our uniqueness and existence results enable us to discuss motion by mean curvature with a variety of initial hypersurfaces including spheres, tori and other compact surfaces. We are particularly interested in the motion of a subelliptic sphere. It turns out that if $u_0$ is a defining function of the sphere centered at $0$ with radius $r$, say
\[
u_0(p)=\min\{(p_1^2+p_2^2)^2+16p_3^2-r^4, M\}
\]
with $p=(p_1, p_2, p_3)\in \mathcal{H}$ and $M>0$ large, then the unique solution of (MCF) is
\[
u(p, t)=\min\{(p_1^2+p_2^2)^2+12 t(p_1^2+p_2^2)+16p_3^2+12t^2-r^4, M\}
\]
for any $t\geq 0$. We need to truncate the initial function and the solution by a constant $M$ because all of our wellposedness results are for solutions that are constant outside a compact set. It is obvious that the zero level set $\Gamma_t$ of $u$ vanishes after time $t=r^2/\sqrt{12}$, which, by Theorem \ref{comparison theorem}, indicates that all compact surfaces under the motion by horizontal mean curvature disappear in finite time. 

To understand the \textit{asymptotic profile} at the \textit{extinction time}, we normalize the evolution $\Gamma_t$ initialized from the sphere and find that the normalized surface $\Gamma_t/\sqrt{r^4-12t^2}$ converges to an ellipsoid given by the following equation:
\[
\sqrt{12}r^2(p_1^2+p_2^2)+16p_3^2=1.
\]
The asymptotic profile above depends on $r$, the size of the initial surface, which is quite different from the Euclidean case. 

The paper is organized in the following way. We present an introduction in Section \ref{sec prelim} about the Heisenberg group including calculations of some particular functions we will use later. In Section \ref{sec defi}, we discuss various kinds of definitions of solution to (\ref{mcf1}). We propose a new definition and show its equivalence with the others. An explicit solution related to the evolution of a subelliptic sphere is given at the end of this section. The comparison principle, Theorem \ref{comparison theorem}, is proved in Section \ref{sec comp}. We establish the games and show the existence theorem in Section \ref{sec existence}. Section \ref{sec stability} is devoted to the stability results and Section \ref{sec properties} is intended to show further properties of the evolution including the uniqueness and finite extinction with the interesting asymptotic profile.

\section{Tools from  Calculus in $\mathcal{H}$}
Good references for this section are the course notes \cite{M1} and the monograph
\cite{CDPT}.
\subsection{Preliminaries}\label{sec prelim}
Recall the that Heisenberg group $\mathcal{H}$ is $\mathbb{R}^{3}$ endowed with the non-commutative group multiplication 
\[
(p_1, p_2, p_3)\cd (q_1, q_2, q_3)=\left(p_1+q_1, p_2+q_2, p_3+q_3+\frac{1}{2}(p_1q_2-q_1p_2)\right),
\]
for all $p=(p_{1},p_{2},p_{3})$ and $q=(q_{1},q_{2}.q_{3})$ in $\mathcal{H}$.
The Haar measure if $\mathcal{H}$ is the usual Lebesgue measure in $\mathbb{R}^{3}$.
The Kor\'anyi gauge is given by
\[
|p|=((p_1^2+p_2^2)^2+16 p_3^2)^{1/4},
\]
and the left-invariant Kor\'anyi or gauge metric is 
\[
d(p, q)=|q^{-1}\cdot p|.
\]
The Kor\'anyi ball of radius $r>0$ centered at $p$ is
$$B_r(p):=\{q\in \mathcal{H}\colon d(p,q)<r\}.$$
The Lie Algebra of $\mathcal{H}$ is generated by the left-invariant vector fields
\[
\begin{aligned}
& X_1=\frac{\partial}{\partial p_{1}}-\frac{p_{2}}{2}\frac{\partial}{\partial p_{3}};\\
& X_2=\frac{\partial}{\partial p_{2}}+\frac{p_{1}}{2}\frac{\partial}{\partial p_{3}};\\
& X_3=\frac{\partial}{\partial p_{3}}.
\end{aligned}
\]
  One may easily verify the commuting relation $X_3=[X_1, X_2]=
  X_{1}X_{2}- X_{2}X_{1}$.\par
  For any smooth real valued function $u$ defined in an open subset of $\mathcal{H}$, the horizontal gradient of $u$ is 
 $$\nabla_H u=(X_1 u, X_2 u)$$  while the complete gradient of $u$ is 
  $$\nabla u=(X_1 u, X_2 u, X_3 u).$$ For further details about the relation between sub-Riemannian metrics in Carnot group and Riemaniann metrics see \cite{AFM}.  

The symmetrized second horizontal Hessian  $(\nabla_H^2 u)^\ast$ is the $2\times2$ symmetry matrix given by  
\[
(\nabla_H^2 u)^\ast:=\left(\begin{array}{cc} X_1^2 u & (X_1X_2 u+X_2X_1 u)/2\\
(X_1X_2 u+X_2X_1 u)/2 & X_2^2 u\end{array}\right).
\]
We will also consider the symmetrized complete Hessian $(\nabla^2 u)^\ast$  defined as
the $3\times3$ symmetric matrix 
\[
(\nabla^2 u)^\ast:=
\left(\begin{array}{ccc} X_1^2 u & (X_1X_2 u+X_2X_1 u)/2 & (X_1X_3 u+X_3X_1 u)/2\\
(X_1X_2 u+X_2X_1 u)/2 & X_2^2 u & (X_2X_3 u+X_3X_2 u)/2 \\
(X_1X_3 u+X_3X_1 u)/2 & (X_2X_3 u+X_3X_2 u)/2 & X_3^2 u\end{array}
\right),
\]

\subsection{Derivatives of auxiliary functions}
Here we include several basic calculations for some test functions related to the Kor\'anyi distance, which will be used in the proof of comparison theorem for generalized horizontal mean curvature flow.\par

We are interested in the first and second horizontal derivatives of 
\[
\begin{aligned}
f(p, q):&=d(p, q)^4\\
&=\left((p_1-q_1)^2+(p_2-q_2)^2\right)^2+16\left(p_3-q_3-{1 \over 2}q_1p_2+{1\over 2}q_2p_1\right)^2.\\
\end{aligned}
\] 
We use the super index $p$ to denote derivatives with respect to the $p$ variable and
follow the same convention for derivatives with respect to $q$. 

Let us record the results of our calculation:

\begin{equation}\label{deri-fp1}
\begin{aligned}
X_1^p f(p, q)=4\big((p_1-&q_1)^2+(p_2-q_2)^2\big)(p_1-q_1)\\
&-16(p_2-q_2)\left(p_3-q_3+\frac{1}{2}(q_2p_1-q_1p_2)\right);\\
\end{aligned}
\end{equation}
\begin{equation}\label{deri-fp2}
\begin{aligned}
X_2^p f(p, q)=4\big((p_1-&q_1)^2+(p_2-q_2)^2\big)(p_2-q_2)\\
&+16(p_1-q_1)\left(p_3-q_3+\frac{1}{2}(q_2p_1-q_1p_2)\right);\\
\end{aligned}
\end{equation}

\begin{equation}\label{deri-fq1}
\begin{aligned}
X_1^q f(p, q)=-4\big((p_1-&q_1)^2+(p_2-q_2)^2\big)(p_1-q_1)\\
&-16(p_2-q_2)\left(p_3-q_3+\frac{1}{2}(q_2p_1-q_1p_2)\right);\\
\end{aligned}
\end{equation}

\begin{equation}\label{deri-fq2}
\begin{aligned}
X_2^q f(p, q)=-4\big((p_1-&q_1)^2+(p_2-q_2)^2\big)(p_2-q_2)\\
&+16(p_1-q_1)\left(p_3-q_3+\frac{1}{2}(q_2p_1-q_1p_2)\right);\\
\end{aligned}
\end{equation}

It is clear that in general $\nabla_H^p f(p, q)\neq -\nabla_H^q f(p, q)$, which is not the case in the Euclidean case. But the following Euclidean property still holds here. 
\begin{prop}\label{prop1}
If either $\nabla_H^p \left(|q^{-1}\cd p|^4\right)=0$ or $\nabla_H^q \left(|q^{-1}\cd p|^4\right)=0$, then the horizontal components of $p$ and $q$ are equal, i.e., $p_1=q_1$ and $p_2=q_2$. \end{prop}
\begin{proof}
Set 
\[
\begin{aligned}
&A:=4\big((p_1-q_1)^2+(p_2-q_2)^2\big),\\ 
&B:=16\left(p_3-q_3+\frac{1}{2}(q_2p_1-q_1p_2)\right).
\end{aligned}
\] 
When $\nabla_H^p \left(|q^{-1}\cd p|^4\right)=0$, the calculations (\ref{deri-fp1}) and (\ref{deri-fp2}) read
\begin{equation}\label{linear system}
\left\{
\begin{aligned}
A(p_1-q_1)-B(p_2-q_2)=0;\\
B(p_1-q_1)+A(p_2-q_2)=0
\end{aligned}
\right.
\end{equation}
with $\det\left(\begin{array}{cc} A & -B\\ B & A \end{array}\right)=A^2+B^2\geq 0$. Since $A^2+B^2=0$ implies that $p_i=q_i$ for $i=1, 2$, the desired result is trivial if $A^2+B^2=0$. If the determinant is not zero, then we also obtain $q_1=p_1$ and $q_2=p_2$ by solving the linear system (\ref{linear system}).
The same argument applies to the case when $\nabla_H^q \left(|q^{-1}\cd p|^4\right)=0$.
\end{proof}

We next calculate the second horizontal derivatives. 
\begin{equation}\label{sec deri1}
X_1^{2, p} f(p, q)=X_1^{2, q} f(p, q)= 12(p_1-q_1)^2 + 12(p_2-q_2)^2;
\end{equation}

\begin{equation}\label{sec deri2}
X_2^{2, p} f(p, q)=X_2^{2, q} f(p, q)= 12(p_2-q_2)^2 + 12(p_1-q_1)^2;
\end{equation}

\begin{equation}\label{sec deri3}
X_2^pX_1^p f(p, q)=X_1^qX_2^q f(p, q)
=-16\left(p_3-q_3+\frac{1}{2}(q_2p_1-q_1p_2)\right)=-B;
\end{equation}

\begin{equation}\label{sec deri4}
X_1^pX_2^p f(p, q)=X_2^qX_1^q f(p, q)=16\left(p_3-q_3+\frac{1}{2}(q_2p_1-q_1p_2)\right)=B.
\end{equation}
It is clear that 
\[
\frac{1}{2}(X_1^pX_2^p f+X_2^pX_1^p f)=\frac{1}{2}(X_1^qX_2^q f+X_2^qX_1^q f)=0.
\]

For later use, let us investigate the derivatives of another function. Take 
\begin{equation}
\begin{aligned}
&g(p, q):=|p\cd q^{-1}|^4\\
&=\left((p_1-q_1)^2+(p_2-q_2)^2\right)^2+16\left(p_3-q_3-{1 \over 2}p_1q_2+{1\over 2}p_2q_1\right)^2.\\
\end{aligned}
\end{equation}
Then 
\begin{equation}\label{deri-gp1}
\begin{aligned}
X_1^pg(p, q)=4((p_1-q_1)^2&+(p_2-q_2)^2)(p_1-q_1)\\
&-16(p_2+q_2)\left(p_3-q_3-{1\over 2}p_1q_2+{1\over 2}p_2q_1\right);
\end{aligned}
\end{equation}
\begin{equation}\label{deri-gp2}
\begin{aligned}
X_2^pg(p, q)=4((p_1-q_1)^2&+(p_2-q_2)^2)(p_2-q_2)\\
&+16(p_1+q_1)\left(p_3-q_3-{1\over 2}p_1q_2+{1\over 2}p_2q_1\right);
\end{aligned}
\end{equation}
\begin{equation}\label{deri-gq1}
\begin{aligned}
X_1^qg(p, q)=-4((p_1-q_1)^2&+(p_2-q_2)^2)(p_1-q_1)\\
&+16(p_2+q_2)\left(p_3-q_3-{1\over 2}p_1q_2+{1\over 2}p_2q_1\right);
\end{aligned}
\end{equation}
\begin{equation}\label{deri-gq2}
\begin{aligned}
X_2^qg(p, q)=-4((p_1-q_1)^2&+(p_2-q_2)^2)(p_2-q_2)\\
&-16(p_1+q_1)\left(p_3-q_3-{1\over 2}p_1q_2+{1\over 2}p_2q_1\right).
\end{aligned}
\end{equation}
\begin{rmk}\label{oppo grad}
In this case, we do have $\nabla_H^p g(p, q)= -\nabla_H^q g(p, q)$. But the property  in Proposition \ref{prop1} does not hold  in general.
\end{rmk}

The second derivatives are given below.
\begin{equation}\label{sec deri-g1}
X_1^{2, p}g(p, q)=X_1^{2, q}g(p, q)=12(p_1-q_1)^2+4(p_2-q_2)^2+8(p_2+q_2)^2;
\end{equation}

\begin{equation}\label{sec deri-g2}
X_2^{2, p}g(p, q)=X_2^{2, q}g(p, q)=4(p_1-q_1)^2+12(p_2-q_2)^2+8(p_1+q_1)^2;
\end{equation}

\begin{equation}\label{sec deri-g3}
\begin{aligned}
X_1^pX_2^p g(p, q)=X_2^qX_1^q g(p, q)=&8(p_1-q_1)(p_2-q_2)-8(p_1+q_1)(p_2+q_2)\\
&+16(p_3-q_3-{1\over 2}p_1q_2+{1\over 2}p_2q_1);
\end{aligned}
\end{equation}

\begin{equation}\label{sec deri-g4}
\begin{aligned}
X_2^pX_1^p g(p, q)=X_1^qX_2^q g(p, q)=&8(p_1-q_1)(p_2-q_2)-8(p_1+q_1)(p_2+q_2)\\
&-16(p_3-q_3-{1\over 2}p_1q_2+{1\over 2}p_2q_1);
\end{aligned}
\end{equation}

\begin{equation}\label{sec deri-g5}
\begin{aligned}
\frac{1}{2}(X_1^pX_2^p g+X_2^pX_1^p g)&=\frac{1}{2}(X_1^qX_2^q g+X_2^qX_1^q g)\\
&=8(p_1-q_1)(p_2-q_2)-8(p_1+q_1)(p_2+q_2).
\end{aligned}
\end{equation}

\subsection{Extrema in the Heisenberg group}

As $|p|^2\approx p_1^2+p_2^2+|p_3|$ in Heisenberg group, the Taylor formula reads
\begin{equation}\label{Taylor heisenberg}
u(p)=u(\hat{p})+\langle \hat{p}^{-1}\cdot p, \nabla u(\hat{p})\rangle+{1\over 2}\langle (\nabla_H^2 u)^\ast(\hat{p}) h, h\rangle+o(|\hat{p}^{-1}\cd p|^2),
\end{equation}
where $h=(p_1-\hat{p}_1, p_2-\hat{p}_2)$ is the horizontal projection of $\hat{p}^{-1}\cd p$.

The following proposition follows easily from the Euclidean analog.
\begin{prop}[Maxima on Heisenberg group]\label{maximality}
Suppose $\mathcal{O}$ is an open subset of $\mathcal{H}$. Let $u\in C^2(\mathcal{O})$ and $\hat{p}\in \mathcal{O}$. If $u(p)\leq u(\hat{p})$ for all $p\in \mathcal{O}$, then $\nabla u(\hat{p})=0$ and $(\nabla_H^2 u)^\ast(\hat{p})\leq 0$. \par
Analogously, for minima we have that if $u(p)\ge u(\hat{p})$ for all $p\in \mathcal{O}$, then $\nabla u(\hat{p})=0$ and $(\nabla_H^2 u)^\ast(\hat{p})\ge 0$. 
 
\end{prop}

\section{Definitions of solutions}\label{sec defi}

\subsection{General definitions} For a vector $\eta\in\mathbb{R}^{2}$ and
a $2\times2$ symmetric matrix $Y\in \mathbf{S}^2$
we define
\[
F(\eta, Y)=-\tr\left(\left(I-\frac{\eta \otimes \eta}{|\eta|^2}\right)Y\right).
\]
In any open subset $\mathcal{O}\subset \mathcal{H}\times (0, \infty)$ the mean curvature flow equation
\begin{equation}\label{mcf}
u_t-\tr\left[\bigg(I-\frac{\nabla_H u\otimes \nabla_H u}{|\nabla_H u|^2}\bigg)(\nabla_H^2 u)^\ast\right]=0 \ \text{ in $\mathcal{O}$}
\end{equation}
can be written  as
\[
u_t+F(\nabla_H u, (\nabla_H^2 u)^\ast)=0 \ \text{ in $\mathcal{O}$.}
\]

We next define the semicontinuous envelopes in the following way: for any function $h$ defined on a set $\mathcal{O}$ of a metric space $\mathcal{M}$ with values in $\mathbb{R}\cup\{\pm\infty\}$, we take
\begin{equation}\label{upper enve}
h^\star(x)=\lim_{r\to 0}\sup\{h(y): y\in \mathcal{O}\cap B_r(x)\}
\end{equation}
and 
\begin{equation}\label{lower enve}
h_\star(x)=\lim_{r\to 0}\inf\{h(y): y\in \mathcal{O}\cap B_r(x)\}
\end{equation}
for any $x\in \ol{\mathcal{O}}$, where $B_r(x)$ denotes the ball with radius $r>0$ centered at $x$. It is easily seen that 
\[
F^\star(0, 0)=F_\star(0, 0)=0;
\]
\[
F^\star(\eta, X)=F_\star(\eta, X)=F(\eta, X) \text{ for all $(\eta, X)\in\mathbb{R}^2\setminus\{0\}\times\mathbf{S}^2$.}
\]

One type of definition of viscosity solutions of (\ref{mcf}) is as follows.
\begin{defi}\label{enve defi}
An upper (resp., lower) semicontinuous function $u$ defined on $\mathcal{O}\subset\mathcal{H}\times (0, \infty)$ is a subsolution (resp., supersolution) of (\ref{mcf}) if 
\begin{enumerate}
\item[(\rmnum{1})] $u<\infty$ (resp., $u>-\infty$) in $\mathcal{O}$;
\item[(\rmnum{2})] for any smooth function $\phi$ such that
\[
\max_{\mathcal{O}}u-\phi=(u-\phi)(\hat{p}, \hat{t}),
\]
\[
\text{(resp., }
\min_{\mathcal{O}}u-\phi=(u-\phi)(\hat{p}, \hat{t}),)
\]   
it satisfies 
\[
\phi_t+F_\star(\nabla_H \phi, (\nabla_H^2 \phi)^\ast)\leq 0 \text{ at $(\hat{p}, \hat{t})$,}
\]
\[
\text{(resp.,  }
\phi_t+F^\star(\nabla_H \phi, (\nabla_H^2 \phi)^\ast)\geq 0 \text{ at $(\hat{p}, \hat{t})$).}
\]
\end{enumerate}
A function $u$ is called a solution of (\ref{mcf}) if it is both a subsolution and a supersolution.
\end{defi}

We now propose another definition for the horizontal mean curvature flow equation following Giga \cite{G1}.
\begin{defi}\label{trad defi}
An upper (resp., lower) semicontinuous function $u$ defined on $\mathcal{O}\subset\mathcal{H}\times (0, \infty)$ is a subsolution (resp., supersolution) of (\ref{mcf})
if 
\begin{enumerate}
\item[(\rmnum{1})] $u<\infty$ (resp., $u>-\infty$) in $\mathcal{O}$;
\item[(\rmnum{2})] for any smooth function $\phi$ such that
\[
\max_{\mathcal{O}}u-\phi=(u-\phi)(\hat{p}, \hat{t}),
\]
\[
\text{(resp., }
\min_{\mathcal{O}}u-\phi=(u-\phi)(\hat{p}, \hat{t}),)
\]    
it satisfies 
\[
\phi_t+F(\nabla_H \phi, (\nabla_H^2 \phi)^\ast)\leq 0 \text{ at $(\hat{p}, \hat{t})$,}
\]
\[
\text{(resp.,  }
\phi_t+{F}(\nabla_H \phi, (\nabla_H^2 \phi)^\ast)\geq 0 \text{ at $(\hat{p}, \hat{t})$,)}
\]
when $\nabla_H \phi(\hat{p}, \hat{t})\neq 0$ and 
\[
\phi_t(\hat{p}, \hat{t})\leq 0, 
\]
\[
\text{(resp., }
\phi_t(\hat{p}, \hat{t})\geq 0, )
\]
when $\nabla_H \phi(\hat{p}, \hat{t})= 0$ and $(\nabla_H^2\phi)^\ast(\hat{p}, \hat{t})=0$.
\end{enumerate}
A function $u$ is called a solution of (\ref{mcf}) if it is both a subsolution and a supersolution.
\end{defi}
\begin{rmk}
One may replace the maximum (resp., minimum) in condition (\rmnum{2}) of the above definitions with a strict maximum by adding a positive (resp., negative) smooth gauge to $\phi$.
\end{rmk}

The  definition  using subelliptic semijets  is as follows. 
\begin{defi}\label{jet defi}
An upper (resp., lower) semicontinuous function $u$ defined on $\mathcal{O}\subset\mathcal{H}\times (0, \infty)$ is a subsolution (resp., supersolution) of (\ref{mcf}) if 
\begin{enumerate}
\item $u<\infty$ (resp., $u>-\infty$) in $\mathcal{O}$;
\item for any $(\tau, \eta, \mathcal{X})\in \ol{J}_H^{2, +}u(\hat{p}, \hat{t})$ (resp., $(\tau, \eta, \mathcal{X})\in \ol{J}_H^{2, -}u(\hat{p}, \hat{t})$) with $(\hat{p}, \hat{t})\in \mathcal{O}$, we have
\[
\phi_t+F_\star(\nabla_H \phi, (\nabla_H^2 \phi)^\ast)\leq 0 \text{ at $(\hat{p}, \hat{t})$,}
\]
\[
\text{(resp.,  }
\phi_t+F^\star(\nabla_H \phi, (\nabla_H^2 \phi)^\ast)\geq 0 \text{ at $(\hat{p}, \hat{t})$,)}
\]

\end{enumerate}
A function $u$ is called a solution of (\ref{mcf}) if it is both a subsolution and a supersolution.
\end{defi}
It is not hard to see that Definition \ref{jet defi} is equivalent to Definition \ref{enve defi}.
Roughly speaking, in Definition \ref{trad defi} and Definition \ref{jet defi} we restrict the test function space to the following 
\[
\mathcal{A}_0=\{\phi\in C^\infty(\mathcal{H}): \nabla_H\phi(p)=0 \text{ implies } (\nabla^2_H)^\ast \phi(p)=0\}.
\]
The next result, which is actually a variant of \cite[Proposition 2.2.8]{G1} for the Heisenberg group, indicates the equivalence between this new definition and the known one in spite of the restriction on the test functions.
\begin{prop}[Equivalence of definitions]
An upper (resp., lower) semicontinuous function $u:\mathcal{O}\to \mathbb{R}$ is a subsolution (resp., supersolution) of (\ref{mcf}) defined as in Definition \ref{trad defi} (in $\mathcal{O}$) if and only if it is a subsolution (resp., superolution) in $\mathcal{O}$ in the sense of Definition \ref{enve defi}.
\end{prop}
\begin{proof}
It is obvious that Definition \ref{trad defi} is a relaxation of Definition \ref{enve defi}. We prove the reverse implication only for subsolutions. The statement for supersolutions can be proved similarly. Suppose there are a smooth function $\phi$ and $(\hat{p}, \hat{t})\in \mathcal{O}$ such that 
\[
\max_\mathcal{O} (u-\phi)=(u-\phi)(\hat{p}, \hat{t})
\]
By usual modification in the definition of viscosity solutions, we may assume it is a strict maximum. 
We construct 
\[
\Psi_\varepsilon(p, q, t):=u(p, t)-\frac{1}{\varepsilon}|q^{-1}\cd p|^4-\phi(q, t).
\]
It is clear that 
\[
\Psi^\ast(p, q, t):=\limsups_{\varepsilon\to 0}\Psi_\varepsilon (p, q, t)=\begin{cases} u(p, t)-\phi(p, t) &\text{ if $p=q$}\\ -\infty &\text{ if $p\neq q$}\end{cases}
\] 
attains a strict maximum at $(\hat{p}, \hat{p}, \hat{t})$. By the convergence of maximizers (\cite[Lemma 2.2.5]{G1}), we may take $p^\varepsilon, q^\varepsilon, t^\varepsilon$ converging to $\hat{p}, \hat{p}, \hat{t}$ respectively as $\varepsilon\to 0$ such that
$\Psi_\varepsilon$ attains a maximum at $(p^\varepsilon, q^\varepsilon, t^\varepsilon)$. It follows that $q\mapsto -\frac{1}{\varepsilon}|q^{-1}\cd p^\varepsilon|^4-\phi(q, t)$ has a maximum at $q^\varepsilon$, which, by Proposition \ref{maximality}, implies that
\begin{equation}\label{equiv1}
-{1\over \varepsilon}\nabla_H^q f(p^\varepsilon, q^\varepsilon)=\nabla_H\phi(q^\varepsilon, t^\varepsilon);
\end{equation}
\begin{equation}\label{equiv2}
-{1\over \varepsilon}(\nabla_H^{2, q}f)^\ast(p^\varepsilon, q^\varepsilon)\leq (\nabla_H^2\phi)^\ast(q^\varepsilon, t^\varepsilon),
\end{equation}
where $f(p, q)=|q^{-1}\cd p|^4$.

We next discuss the following two cases.\\
\textbf{Case A.} $\nabla_H\phi(q^\varepsilon, t^\varepsilon)\neq 0$ for a subsequence of $\varepsilon\to 0$. (We still use $\varepsilon$ to denote the subsequence.) \\
Since the maximality of $\Psi$ at $(p^\varepsilon, q^\varepsilon, t^\varepsilon)$ implies that
\[
(p, t)\mapsto u(p, t)-{1\over \varepsilon}f(p^\varepsilon, q^\varepsilon)-\phi(p\cd (p^\varepsilon)^{-1}\cd q^\varepsilon, t)
\]
attains a maximum at $(p^\varepsilon, t^\varepsilon)\in \mathcal{O}$. Denote $\phi^\varepsilon(p, t)=\phi(p\cd (p^\varepsilon)^{-1}\cd q^\varepsilon, t)$. We apply Definition \ref{trad defi} to get
\begin{equation}
\phi_t+F(\nabla_H\phi^\varepsilon, (\nabla_H^2\phi^\varepsilon)^\ast)\leq 0 \text{ at $(p^\varepsilon, t^\varepsilon)$}
\end{equation}
Since the derivative of the right multiplication tends to $0$ as $\varepsilon\to 0$ and its second derivatives are $0$, we have
\[
\nabla_H\phi^\varepsilon(p^\varepsilon, t^\varepsilon)\to \nabla_H\phi(\hat{p}, \hat{t}) \text{ and } (\nabla_H^2\phi^\varepsilon)^\ast(p^\varepsilon, t^\varepsilon)\to (\nabla_H^2\phi)^\ast(\hat{p}, \hat{t}) \text{ as $\varepsilon\to 0$}. 
\]
It follows immediately that
\[
\phi_t+F_\star(\nabla_H \phi, (\nabla_H^2 \phi)^\ast)\leq 0 \text{ at $(\hat{p}, \hat{t})$.}
\]

\noindent \textbf{Case B.} $\nabla_H\phi(q^\varepsilon, t^\varepsilon)= 0$ for all sufficiently small $\varepsilon>0$. \\
It follows from (\ref{equiv1}) that $\nabla_H^q f(p^\varepsilon, q^\varepsilon)=0$,  
which by Proposition \ref{prop1} yields that
\begin{equation}\label{equal deri}
p^\varepsilon_i=q^\varepsilon_i \text{ for $i=1, 2$.}
\end{equation}
In terms of (\ref{deri-fp1})-(\ref{deri-fp2}) and (\ref{sec deri1})--(\ref{sec deri4}), we have 
\begin{equation}
\nabla_H^p f(p^\varepsilon, q^\varepsilon)=0 \text{ and } \nabla_H^{2, p}f(p^\varepsilon, q^\varepsilon)=0. 
\end{equation}
Since $(p^\varepsilon, t^\varepsilon)$ is a maximizer of 
\[
(p, t)\mapsto u(p, t)-{1\over \varepsilon}f(p, q^\varepsilon)-\phi(q^\varepsilon, t),
\]
applying Definition \ref{trad defi} and sending the limit, we obtain
\begin{equation}\label{equiv3}
\phi_t(\hat{p}, \hat{t})\leq 0,
\end{equation}
On the other hand, by passing to the limit in (\ref{equiv1}) and (\ref{equiv2}), we have
\begin{equation}\label{phi deri1}
\nabla_H \phi(\hat{p}, \hat{t})=0
\end{equation} 
and 
\begin{equation}\label{phi deri2}
\nabla_H^2\phi(\hat{p}, \hat{t})\geq 0.
\end{equation}
By (\ref{phi deri1}),  (\ref{equiv3}) is equivalent to 
\[
\phi_t(\hat{p}, \hat{t})+F_\star(\nabla_H \phi(\hat{p}, \hat{t}), 0)\leq 0,
\]
which, thanks to (\ref{phi deri2}) and the ellipticity of $F$, implies that 
\[
\phi_t(\hat{p}, \hat{t})+F_\star(\nabla_H\phi(\hat{p}, \hat{t}), \nabla_H^2\phi(\hat{p}, \hat{t}))\leq 0.
\]

\end{proof}

\subsection{An explicit solution} We provide an example of solutions of (\ref{mcf}) when the initial value is the fourth power of a smooth gauge of the Heisenberg group. We can actually express a solution explicitly. 
\begin{prop} \label{prop explicit solution}
For any $p=(p_1, p_2, p_3)\in \mathcal{H}$, let 
\begin{equation}\label{gauge}
G(p)=|p|^4=(p_1^2+p_2^2)^2+16p_3^2.
\end{equation}
Then 
\begin{equation}\label{explicit solution}
w(p, t)=(p_1^2+p_2^2)^2+12 t(p_1^2+p_2^2)+16p_3^2+12t^2
\end{equation}
is a continuous solution of (\ref{mcf1}) and $w(p, 0)=G(p)$.
\end{prop}
\begin{proof}
Since $w$ is smooth, the proof is based on a straightforward calculation of the first derivatives of $w$
\begin{equation}\label{first deri explicit}
\begin{aligned}
&w_t=12(p_1^2+p_2^2)+24t,\\
&X_1 w=Kp_1-16p_2p_3, \ X_2 w=Kp_2+16p_1p_3,
\end{aligned}
\end{equation}
where $K:=4(p_1^2+p_2^2)+24t$ and the second derivatives
\begin{equation}
\begin{aligned}
&X_1^2 w=X_2^2 w=12 p_1^2+12 p_2^2+24t,\\
&X_1X_2w=16p_3, \ X_2X_1w=-16p_3,\\
&(\nabla_H^2 w)^\ast=\left(\begin{array}{cc}
12 p_1^2+12 p_2^2+24t & 0\\
0 & 12 p_1^2+12 p_2^2+24t
\end{array}\right).
\end{aligned}
\end{equation}
Noting that $(\nabla_H^2 w)^\ast$ is constant multiple of the identity, we easily conclude from our calculation that
\[
\begin{aligned}
&F^\star(\nabla_H w, \nabla_H^2 w)=F_\star(\nabla_H w, \nabla_H^2 w)\\
=&\tr\left[\bigg(I-\frac{\nabla_H w\otimes \nabla_H w}{|\nabla_H w|^2}\bigg)(\nabla_H^2 w)^\ast\right]=12(p_1^2+p_2^2)+24t=w_t,
\end{aligned}
\]
which means that $w$ satisfies (\ref{mcf}) by Definition \ref{enve defi}.
\end{proof}
\begin{rmk}
There is another way to understand that $w$ is a solution of (\ref{mcf}) by adopting Definition \ref{trad defi} when $\nabla_H w=0$ at $(p, t)\in \mathcal{H}\times (0, \infty)$. If $\nabla_Hw(p, t)=0$, we have $p_1=p_2=0$ by solving a linear system 
\[
\left(\begin{array}{cc}K & -16p_3\\ 16p_3 & K\end{array}\right)
\left(\begin{array}{c} p_1\\ p_2\end{array}\right)=
\left(\begin{array}{c} 0\\ 0\end{array}\right)
\]
with 
\[
\det\left(\begin{array}{cc}K & -16p_3\\ 16p_3 & K\end{array}\right)=K^2+16p_3^2>0.
\]
In addition, 
\[
(\nabla_H^2 w)^\ast=\left(\begin{array}{cc}
24t & 0\\
0 & 24t
\end{array}\right).
\]
Note that, by Proposition \ref{maximality}, it is not possible to take a smooth function $\phi$ touching $w$ from above at $(p, t)$ with 
\begin{equation}\label{testing condition}
\nabla_H\phi(p, t)=0 \text{ and } (\nabla_H^2\phi)^\ast(p, t)=0.
\end{equation} 
Therefore $w$ is  a subsolution of (\ref{mcf1}) at $(p, t)$ by Definition \ref{trad defi}. On the other hand, whenever a test function $\phi$ touches $w$ from below at $(p, t)$ with (\ref{testing condition}), we get $\phi_t(p, t)=w_t(p, t)=24t>0$, which implies that $w$ is also a supersolution due to Definition \ref{trad defi}.
\end{rmk}
\begin{rmk}\label{rmk tran free}
A basic transformation keeps the solution (\ref{explicit solution}) being a solution. To be more precise, for any fixed $c\in\mathbb{R}$, $L>0$ and $\hat{p}\in \mathcal{H}$, we define $\hat{w}(p, t)=Lw(\hat{p}^{-1}\cdot p, t)+c$ for all $(p, t)\in \mathcal{H}\times [0, \infty)$. Then we claim that $\hat{w}$ is a solution of (\ref{mcf1}). Indeed, our calculation above extends to
\[
\begin{aligned}
&X_1^2 \hat{w}=X_2^2\hat{w}=12L(p_1-\hat{p}_1)^2+12L(p_2-\hat{p}_2)^2+24Lt;\\
&X_1X_2 \hat{w}=-X_2X_1\hat{w}=16L(p_3-\hat{p}_3-{1\over 2}\hat{p}_1p_2+{1\over 2}p_1\hat{p}_2).
\end{aligned}
\]
The conclusion follows immediately as in the proof of Proposition \ref{prop explicit solution}.
\end{rmk}

A primary and geometric observation for the explicit solution $u$ in (\ref{explicit solution}) is as follows. For any fixed $\mu>0$, the $\mu$-level set,
\[
\Gamma^\mu_t=\{p\in \mathcal{H}: w(p, t)=\mu\}
\]
describes the position of surface at time $t\geq 0$. It is obvious that even if $\Gamma^\mu_0\neq \emptyset$, $\Gamma_t$ will vanish when $t$ is sufficiently large, which agrees with the usual extinction of mean curvature flows. We will revisit this property in Section \ref{sec extinction}.

A natural question now is whether the explicit solution we found is the only solution of (MCF) with the initial data (\ref{gauge}). This is related to the open question on the uniqueness of solutions of (MCF). In the following sections we will give an affirmative answer for the case when the initial data are cylindrically symmetric about the vertical axis.


\section{Comparison principle}\label{sec comp}

\subsection{Cylindrically symmetric solutions}
Before presenting the proof of Theorem \ref{comparison theorem}, let us investigate the properties for the solutions of (MCF) that are axisymmetric with respect to the vertical axis; in other words, we consider solutions of the form $u=u(r, z, t)$ where $r=(x^2+y^2)^{1/2}$. 

\begin{lem}[Tests for axisymmetric solutions]\label{lem axis-test}
Let $u$ be a subsolution (resp., supersolution) of (\ref{mcf}). Suppose that there exists $(\hat{p}, \hat{t})\subset \mathcal{H}\times (0, \infty)$ and $\phi\in C^2(\mathcal{O})$ such that 
\[
\max_{\mathcal{O}} (u-\phi)=(u-\phi)(\hat{p}, \hat{t})\quad   (\text{resp., } \min_{\mathcal{O}} (u-\phi)=(u-\phi)(\hat{p}, \hat{t})).
\]
If $\hat{p}=(\hat{p}_1, \hat{p_2}, \hat{p_3})$ satisfies $\hat{p}_1^2+\hat{p}_2^2\neq 0$ and $u$ is axisymmetric about the vertical axis, then there exists $k\in \mathbb{R}$ such that 
\begin{equation}\label{normal test}
{\partial\over\partial p_1}\phi(\hat{p}, \hat{t})=\hat{p_1} k \ \text{ and }\ {\partial\over\partial p_2}\phi(\hat{p}, \hat{t})=\hat{p_2}k.
\end{equation}
\end{lem}
\begin{rmk}
It is clear that $k={\partial\over\partial r}\phi(\sqrt{\hat{p}_1^2+\hat{p}_2^2}, \hat{p}_3, \hat{t})$ provided that $\phi=\phi(r, p_3, t)$, i.e., $\phi$ is also axisymmetric about the vertical axis. 
\end{rmk}
\begin{proof}
Denote $\hat{r}=\sqrt{\hat{p}_1^2+\hat{p}_2^2}$. We only prove the situation when $u$ is a subsolution. By the symmetry of $u$, $u(p_1, p_2, \hat{p}_3, \hat{t})=u(\hat{p}_1, \hat{p}_2, \hat{p}_3, \hat{t})$ for all $p_1^2+p_2^2=\hat{r}^2$. By assumption, we have
\[
(u-\phi)(p_1, p_2, \hat{p}_3, \hat{t})\leq (u-\phi)(\hat{p}_1, \hat{p}_2, \hat{p}_3, \hat{t}) \text{ for all $(p_1, p_2, p_3, t)\in \mathcal{O}$,}
\]
which implies that 
\[
\phi(p_1, p_2, \hat{p}_3, \hat{t})\geq \phi(\hat{p}_1, \hat{p}_2, \hat{p}_3, \hat{t})   
\]
for all $(p_1, p_2)$ close to $(\hat{p}_1, \hat{p}_2)$ with $p_1^2+p_2^2=r^2$. Applying the method of Lagrange's multiplier, we get $k\in \mathbb{R}$ such that
\[
\begin{aligned}
&{\partial\over \partial p_1}\left(\phi(p_1, p_2, \hat{p}_3, \hat{t})-{k\over 2}(p_1^2+p_2^2-\hat{r}^2)\right)=0,\\
&{\partial\over \partial p_2}\left(\phi(p_1, p_2, \hat{p}_3, \hat{t})-{k\over 2}(p_1^2+p_2^2-\hat{r}^2)\right)=0
\end{aligned}
\]
at $(\hat{p}_1, \hat{p}_2)$. We conclude (\ref{normal test}) by straightforward calculations.

\end{proof}

\subsection{Proof of the comparison theorem}
We are now in a position to prove Theorem \ref{comparison theorem}.

\begin{proof}[Proof of Theorem \ref{comparison theorem}]
Let us assume $u$ is axisymmetric about the vertical axis. The same argument applies to the case when $v$ is axisymmetric. Suppose by contradiction that there exists $(\hat{p}, \hat{t})\in \mathcal{H}\times (0, T)$ such that 
\[
(u-v)(\hat{p}, \hat{t})>0
\]
We may assume that $(\hat{p}, \hat{t})$ satisfies 
\begin{equation}\label{max aux0}
u(\hat{p}, \hat{t})-v(\hat{p}, \hat{t})-\frac{\sigma}{T-\hat{t}}=\max_{\mathcal{H}\times [0, T)} \left(u(p, t)-v(p, t)-\frac{\sigma}{T-t}\right)=\mu>0,
\end{equation}
when $\sigma>0$ is small. 
We fix such $\sigma$, double the variables and set up an auxiliary function 
\[
\Phi^\varepsilon(p, t, q, s)=u(p, t)-v(q, s)-{1\over \varepsilon}g^2(p, q)-{1\over 2\varepsilon}(t-s)^2-\frac{\sigma}{T-t},
\]
where $g(p, q)=|p\cd q^{-1}|^4$. Let $(p^\varepsilon, t^\varepsilon, q^\varepsilon, s^\varepsilon)\in (\mathcal{H}\times [0, T))^2$ be a maximizer of $\Phi^\varepsilon$, then it is clear that 
\[
\Phi^\varepsilon(p^\varepsilon, t^\varepsilon, q^\varepsilon, s^\varepsilon)=\sup_{(\mathcal{H}\times [0, T))^2} \Phi^\varepsilon>\Phi^\varepsilon(\hat{p}, \hat{t}, \hat{p}, \hat{t}),
\]
which implies that
\begin{equation}\label{max aux}
{1\over \varepsilon}g^2(p^\varepsilon, q^\varepsilon)+{1\over 2\varepsilon}(t^\varepsilon-s^\varepsilon)^2\leq u(p^\varepsilon, t^\varepsilon)-v(q^\varepsilon, s^\varepsilon)-u(\hat{p}, \hat{t})+v(\hat{p}, \hat{t})+\frac{\sigma}{T-\hat{t}}-\frac{\sigma}{T-t^\varepsilon}.
\end{equation}
By the boundedness of $u$ and $v$, we have 
\[
|p^\varepsilon\cd (q^\varepsilon)^{-1}|\to 0 \text{ and }|t^\varepsilon-s^\varepsilon|\to 0 \text{ as $\varepsilon\to 0$.}
\]
Since $u=a$ and $v=b$ with $a\leq b$ outside $K\times [0, \infty)$, we may take a subsequence of $\varepsilon$, still indexed by $\varepsilon$, such that $p^\varepsilon, q^\varepsilon\to \ol{p}\in \mathcal{H}$ and $t^\varepsilon, s^\varepsilon\to \ol{t}\in [0, T)$ as $\varepsilon\to 0$. Sending the limit in (\ref{max aux}) and applying (\ref{max aux0}), we get
\[
\limsup_{\varepsilon\to 0}\left({1\over \varepsilon}g^2(p^\varepsilon, q^\varepsilon)+{1\over 2\varepsilon}(t^\varepsilon-s^\varepsilon)^2\right)\leq 0.
\]
In other words, we have
\begin{equation}\label{limit aux}
{1\over \varepsilon}g^2(p^\varepsilon, q^\varepsilon)\to 0\text{ and }{1\over 2\varepsilon}(t^\varepsilon-s^\varepsilon)^2\to 0 \text{ as $\varepsilon\to 0$.}
\end{equation}
We next claim that $\ol{t}\neq 0$. Indeed, if $\ol{t}=0$, then, since $u(p, 0)\leq v(p, 0)$ for all $p\in \mathcal{H}$, we are led to
\[
\Phi^\varepsilon(p^\varepsilon, t^\varepsilon, q^\varepsilon, s^\varepsilon)\to u(p, 0)-v(p, 0)-\frac{\sigma}{T}<0,
\]
which contradicts the fact that $\Phi^\varepsilon(p^\varepsilon, t^\varepsilon, q^\varepsilon, s^\varepsilon)\geq \mu$. We next apply the Crandall-Ishii lemma and get
\[
\begin{aligned}
&\left(\frac{\sigma}{(T-t^\varepsilon)^2}+{1\over \varepsilon}(t^\varepsilon-s^\varepsilon), {1\over \varepsilon}\nabla^p g^2(p^\varepsilon, q^\varepsilon), \mathcal{X}^\varepsilon\right)\in \ol{J}_H^{2, +} u(p^\varepsilon, t^\varepsilon);\\
&\left({1\over \varepsilon}(t^\varepsilon-s^\varepsilon), -{1\over\varepsilon}\nabla^q g^2(p^\varepsilon, q^\varepsilon), \mathcal{Y}^\varepsilon\right)\in \ol{J}_H^{2, -} v(q^\varepsilon, s^\varepsilon),
\end{aligned}
\]
where $\ol{J}_H^{2, +}$ and $\ol{J}_H^{2, -}$ denote the closure of the semijets in Heisenberg group and $\mathcal{X}^\varepsilon, \mathcal{Y}^\varepsilon\in \mathbf{S}^2$ satisfy
\[
\langle \mathcal{X}^\varepsilon\xi, \xi\rangle-\langle \mathcal{Y}^\varepsilon\xi, \xi\rangle\leq {C\over \varepsilon}g(p^\varepsilon, q^\varepsilon)|p^\varepsilon\cd (q^\varepsilon)^{-1}|^4|\xi|^2={C\over \varepsilon}g^2(p^\varepsilon, q^\varepsilon)|\xi|^2
\]
for some $C>0$ and all $\xi\in \mathbb{R}^2$. See \cite{B, M1} for more details on the semijets and the Crandall-Ishii lemma on the Heisenberg group. It follows from (\ref{limit aux}) that 
\begin{equation}\label{sec deri comp}
\limsup_{\varepsilon\to 0}\left(\langle \mathcal{X}^\varepsilon\xi, \xi\rangle-\langle \mathcal{Y}^\varepsilon\xi, \xi\rangle\right)\leq 0
\end{equation}
uniformly for all bounded $\xi\in \mathbb{R}^2$. Moreover, as is derived from Remark \ref{oppo grad}, the following gradient relation holds:
\[
{1\over \varepsilon}\nabla_H^p g^2(p^\varepsilon, q^\varepsilon)= -{1\over\varepsilon}\nabla_H^q g^2(p^\varepsilon, q^\varepsilon).
\]
Let $\eta^\varepsilon$ denote ${1\over \varepsilon}\nabla_H^p g(p^\varepsilon, q^\varepsilon)$.

Finally, we adopt Definition \ref{jet defi} to derive a contradiction.  

\noindent\textbf{Case A.} If $\eta^\varepsilon\neq 0$ for all $\varepsilon>0$ small, then
\begin{equation}\label{vis ineq1}
\frac{\sigma}{(T-t^\varepsilon)^2}+{1\over \varepsilon}(t^\varepsilon-s^\varepsilon)+F(\eta^\varepsilon, \mathcal{X}^\varepsilon)\leq 0
\end{equation}
and 
\begin{equation}\label{vis ineq2}
{1\over \varepsilon}(t^\varepsilon-s^\varepsilon)+F(\eta^\varepsilon, \mathcal{Y}^\varepsilon)\geq 0.
\end{equation}
Taking the difference of (\ref{vis ineq1}) and (\ref{vis ineq2}) yields
\[
\frac{\sigma}{(T-t^\varepsilon)^2}\leq \tr(I-\frac{\eta^\varepsilon\otimes \eta^\varepsilon}{|\eta^\varepsilon|^2})(\mathcal{X}^\varepsilon-\mathcal{Y}^\varepsilon).
\]
Passing to the limit as $\varepsilon\to 0$ with an application of (\ref{sec deri comp}), we end up with
\[
\frac{\sigma}{(T-\ol{t})^2}\leq 0,
\]
which is clearly a contradiction.

\noindent\textbf{Case B.} If $\eta^{\varepsilon_j}={1\over \varepsilon}\nabla_H^p g^2(p^\varepsilon, q^\varepsilon)={1\over \varepsilon}\nabla_H^q g^2(p^\varepsilon, q^\varepsilon)=0$ for a subsequence $\varepsilon_j\to 0$, we obtain, by computation, that
\begin{equation}\label{zero gradient1}
\begin{aligned}
&2g(p^\ej, q^\ej)X_1^p g(p^\ej, q^\ej)=2g(p^\ej, q^\ej)\left(\ppx g(p^\ej, q^\ej)-{p_2^\ej\over 2}\ppz g(p^\ej, q^\ej)\right)=0;\\
&2g(p^\ej, q^\ej)X_2^p g(p^\ej, q^\ej)=2g(p^\ej, q^\ej)\left(\ppy g(p^\ej, q^\ej)+{p_1^\ej\over 2}\ppz g(p^\ej, q^\ej)\right)=0
\end{aligned}
\end{equation}
and 
\begin{equation}\label{zero gradient2}
\begin{aligned}
&2g(p^\ej, q^\ej)X_1^q g(p^\ej, q^\ej)=2g(p^\ej, q^\ej)\left(\pqx g(p^\ej, q^\ej)-{q_2^\ej\over 2}\pqz g(p^\ej, q^\ej)\right)=0;\\
&2g(p^\ej, q^\ej)X_2^q g(p^\ej, q^\ej)=2g(p^\ej, q^\ej)\left(\pqy g(p^\ej, q^\ej)+{q_1^\ej\over 2}\pqz g(p^\ej, q^\ej)\right)=0.
\end{aligned}
\end{equation}
We further discuss two sub-cases. \\
\ul{Case 1.} When $g(p^\ej, q^\ej)=0$, we get $p^\ej=q^\ej$, which implies that 
\[
\nabla_H^p g(p^\ej, q^\ej)=\nabla_H^q g(p^\ej, q^\ej)=0.
\] 
Since 
\begin{equation}\label{sec deri-gg}
\begin{aligned}
&X_1^{2, p}g^2=2(X_1^p g)^2+2g X_1^{2, p} g, && X_2^{2, p}g^2=2(X_2^p g)^2+2g X_2^{2, p} g,\\
&X_1^pX_2^p g^2= 2 X_1^p gX_2^p g+2g X_1^pX_2^p g, && X_2^pX_1^p g^2=2X_2^p gX_1^p g+2g X_2^pX_1^p g,
\end{aligned}
\end{equation}
We have $(\nabla_H^{2, p} g^2)^\ast(p^\ej, q^\ej)=0$. Similarly, we can deduce $(\nabla_H^{2, q} g^2)^\ast(p^\ej, q^\ej)=0$.
By Definition \ref{trad defi}, the viscosity inequalities read
\begin{equation}\label{vis ineq3}
\frac{\sigma}{(T-t^{\varepsilon_j})^2}+{1\over {\varepsilon_j}}(t^{\varepsilon_j}-s^{\varepsilon_j})\leq 0
\end{equation}
and 
\begin{equation}\label{vis ineq4}
{1\over \varepsilon_j}(t^{\varepsilon_j}-s^{\varepsilon_j})\geq 0,
\end{equation}
whose difference implies that $\sigma/(T-t^{\varepsilon_j})^2\leq 0$. This is certainly a contradiction. \\
\ul{Case 2.} When $g(p^\ej, q^\ej)\neq 0$, we get $X_1^p g(p^\ej, q^\ej)=X_2^p g(p^\ej, q^\ej)=0$. We first claim that $p_1^\ej=p_2^\ej=0$. Suppose by contradiction that $(p_1^\ej)^2+(p_2^\ej)^2\neq 0$. In terms of Lemma \ref{lem axis-test}, there is $k\in \mathbb{R}$ such that (\ref{zero gradient1}) reduces to
\[
p_1^\ej k-{p^\ej_2\over 2}\ppz g(p^\ej, q^\ej)=0 \text{ and }p_2^\ej k+{p^\ej_1\over 2}\ppz g(p^\ej, q^\ej)=0,
\]
which yields that $k=0$ and 
\[
\ppz g(p^\ej, q^\ej) =p^\ej_3-q^\ej_3-{1\over 2}p^\ej_1q^\ej_2+{1\over 2}p^\ej_2q^\ej_1=0.
\]
It follows from (\ref{zero gradient1}), (\ref{deri-gp1}) and (\ref{deri-gp2}) that $p^\ej=q^\ej$, which contradicts the assumption that $g(p^\ej, q^\ej)\neq 0$. This completes the proof of our claim.

As $p_1^\ej=p_2^\ej=0$, we apply (\ref{zero gradient1}), (\ref{deri-gp1}) and (\ref{deri-gp2}) again and get
\[
\begin{aligned}
&4((q_1^\ej)^2+(q_2^\ej)^2)(-q_1^\ej)-16q_2^\ej(p_3^\ej-q_3^\ej)=0;\\
&4((q_1^\ej)^2+(q_2^\ej)^2)(-q_2^\ej)+16q_1^\ej(p_3^\ej-q_3^\ej)=0.
\end{aligned}
\]
We are then led to $q_1^\ej=q_2^\ej=0$. Now simplifying the second derivatives of $g^2$ in (\ref{sec deri-gg}) by using (\ref{sec deri-g1})--(\ref{sec deri-g5}), we obtain $(\nabla_H^{2, p} g^2)^\ast(p^\ej, q^\ej)=0$. An analog of calculation yields that $(\nabla_H^{2, q} g^2)^\ast(p^\ej, q^\ej)=0$. The proof is complete since Definition \ref{trad defi} can be adopted once again to get (\ref{vis ineq3})--(\ref{vis ineq4}) and deduce a contradiction.
\end{proof}

\section{Existence theorem by games}\label{sec existence}
The game setting is as follows. A marker, representing the \textit{game state}, is initialized at a state $p\in \mathcal{H}$ from time $0$.
The maturity time given is denoted by $t$. Let the step size for space be
$\varepsilon>0$. Time $\varepsilon^2$ is consumed for every step.
Then the total number of game steps $N$ can be regarded
as $\displaystyle [ t/\varepsilon^2]$. The game states for all steps are denoted in order by 
$\zeta^0, \zeta^1, \dots, \zeta^N$ with $\zeta^0=p$. Two players, Player I and Player II participate the game. Player I intends to minimize at the final state an \textit{objective function}, which in our case is $u_0:\mathcal{H}\to \mathbb{R}$, while Player II is to maximize it. At the $(k+1)$-th round ($k<N$),
\begin{itemize}
\item[(1)] Player I chooses in $\mathcal{H}$ a unit horizontal vector $v^k$, i.e., $v^k=(v^k_1, v^k_2, 0)$ satisfying $|v^k|^2=(v^k_1)^2+(v^k_2)^2=1$. We denote by $S^1_h$ the set of all unit horizontal vectors. 
\item[(2)] Carol has the right to reverse Paul's choice, which
   determines $b^k=\pm 1$;
\item[(3)] The marker is moved from the present state $\zeta^k$ to
$\zeta^k\cdot(\sqrt{2}\varepsilon b^kv^k)$.
\end{itemize}

Then the \textit{state equation} is written inductively as
\begin{equation}\label{state equation}
\left\{
\begin{aligned}
&\zeta^{k+1}=\zeta^k\cdot (\sqrt{2}\varepsilon b^kv^k), \quad k=0, 1, \ldots, N-1;\\
&\zeta^0=p.
\end{aligned}
\right.
\end{equation}
The \textit{value function} is defined to be 
\begin{equation}\label{eq:valfun1}
u^\varepsilon(p, t):=\min_{v^1}\max_{b^1}\ldots\min_{v^N}\max_{b^N}u_0(\zeta^N),
\end{equation}
By the \textit{dynamic programming}:
\begin{equation}\label{dpp}
u^\varepsilon(p, t)=\min_{v\in S^1_h}\max_{b=\pm 1}
u^\varepsilon\left(p\cdot(\sqrt{2}\varepsilon bv), t-\varepsilon^2\right)
\end{equation}
with $u^\varepsilon(p, 0)=u_0(p)$.

Our main result of this section is given below.
\begin{thm}[Existence theorem by games]\label{game approximation}
Assume that $u_0$ is uniformly continuous function in $\mathcal{H}$ and is constant $C\in \mathbb{R}$ outside a compact set. Assume also that $u_0$ is spatially axisymmetric about the vertical axis. Let $u^\varepsilon$ be the
value function defined as in (\ref{eq:valfun1}). Then $u^\varepsilon$ converges, as $\varepsilon\to 0$, to the unique axisymmetric
viscosity solution of (MCF) uniformly on compact subsets of $\mathcal{H}\times [0, \infty)$. Moreover, $u=C$ in $(\mathcal{H}\setminus K)\times (0, \infty)$ for some compact set $K\subset \mathcal{H}$.
\end{thm}

Before presenting the proof of Theorem \ref{game approximation}, we first give bounds for the game trajectories under some particular strategies.
\begin{lem}[Lower bound of the game trajectories]\label{traj lower bound}
For any $p\in \mathcal{H}$ and $t\geq 0$ with $N=[t/\varepsilon^2]$, let $\zeta_k$ be defined as in (\ref{state equation}) for all $k=0, 1, \dots, N$. Then the following statements hold.
\begin{enumerate}
\item[(\rmnum{1})] There exists a strategy of Player I such that 
\begin{equation}\label{traj bound1}
(|\zeta^N_1|^2+|\zeta^N_2|^2)^2+16|\zeta^N_3|^2\geq (|p_1|^2+|p_2|^2)^2+16|p_3|^2
\end{equation}
under this strategy regardless of Player II's choices.
\item[(\rmnum{2})] There exists a strategy of Player II such that (\ref{traj bound1}) holds
under this strategy regardless of Player I's choices.
\end{enumerate}
\end{lem}
\begin{proof}
(i) By direct calculation, we have 
\begin{equation}\label{gauge iteration1}
\begin{aligned}
&((p_1+\sqrt{2}\varepsilon bv_1)^2+(p_2+\sqrt{2}\varepsilon bv_2)^2)^2+16(p_3+{1\over 2}\sqrt{2}\varepsilon b(p_1v_2-p_2v_1))^2\\
=&(p_1^2+p_2^2+2\varepsilon^2)^2+8\varepsilon^2(p_1^2+p_2^2)+16p_3^2\\
&+4\sqrt{2}\varepsilon b( (p_1^2+p_2^2+2\varepsilon^2)(p_1v_1+p_2v_2)+4(p_1p_3v_2-p_2p_3v_1))
\end{aligned}
\end{equation}
It is clear that Player I may take $v=(v_1, v_2, 0)\in S^1_h$ satisfying 
\[
\begin{aligned}
&v_1={1\over \rho}((p_1^2+p_2^2+2\varepsilon^2)p_2+4p_1p_3);\\
&v_2=-{1\over \rho}((p_1^2+p_2^2+2\varepsilon^2)p_1+4p_2p_3),
\end{aligned}
\]
with 
\[
\rho=(p_1^2+p_2^2)^{1/2}((p_1^2+p_2^2+2\varepsilon^2)^2+16p_3^2)^{1/2}
\]
so that, no matter which $b$ is picked, we have
\[
b( (p_1^2+p_2^2+2\varepsilon^2)(p_1v_1+p_2v_2)+4(p_1p_3v_2-p_2p_3v_1))=0
\]
and, furthermore by (\ref{gauge iteration1}), 
\begin{equation}\label{gauge iteration2}
\begin{aligned}
&((p_1+\sqrt{2}\varepsilon bv_1)^2+(p_2+\sqrt{2}\varepsilon bv_2)^2)^2+16(p_3+{1\over 2}\sqrt{2}\varepsilon b(p_1v_2-p_2v_1))^2\\
=&(p_1^2+p_2^2+2\varepsilon^2)^2+8\varepsilon^2(p_1^2+p_2^2)+16p_3^2\geq (p_1^2+p_2^2)^2+16p_3^2.
\end{aligned}
\end{equation}
We can iterate (\ref{gauge iteration2}) to get 
\[
(|\zeta^{k}_1|^2+|\zeta^k_2|^2)^2+16|\zeta^k_3|^2\geq (|\zeta^{k-1}_1|^2+|\zeta^{k-1}_2|^2)^2+16|\zeta^{k-1}_3|^2
\]
for all $k=1, 2, \dots, N$ and (\ref{traj bound1}) follows easily. \\
(\rmnum{2}) The proof of (\rmnum{2}) is similar and even easier. Note that Player II may take a proper $b=\pm1$ so that
\[
b( (p_1^2+p_2^2+2\varepsilon^2)(p_1v_1+p_2v_2)+4(p_1p_3v_2-p_2p_3v_1))\geq 0
\]
and therefore (\ref{gauge iteration2}) holds immediately. We then complete the proof by iteration again.
\end{proof}

\begin{lem}[Upper bound of the game trajectories]\label{traj upper bound}
For any $p\in \mathcal{H}$ and $t\geq 0$ with $N=[t/\varepsilon^2]$, let $\zeta^k$ be defined as in (\ref{state equation}) for all $k=0, 1, \dots, N$. Then the following statements hold.
\begin{enumerate}
\item[(\rmnum{1})] There exists a strategy of Player I such that 
\begin{equation}\label{traj bound2}
(|\zeta^N_1|^2+|\zeta^N_2|^2)^2+16|\zeta^N_3|^2\leq (|p_1|^2+|p_2|^2+6N\varepsilon^2)^2+16|p_3|^2
\end{equation}
under this strategy regardless of Player II's choices.
\item[(\rmnum{2})] There exists a strategy of Player II such that (\ref{traj bound2}) holds
under this strategy regardless of Player I's choices.
\end{enumerate}
\end{lem}
\begin{rmk}
With the notation of the gauge $G$ in (\ref{gauge}), the inequality (\ref{traj bound2}) can be simplified into
\[
G(\zeta^N)\leq (|p_1|^2+|p_2|^2+6t)^2+16|p_3|^2,
\]
which is intuitively natural, since the explicit solution given in (\ref{explicit solution}) satisfies
\[
w(p, t)\leq (|p_1|^2+|p_2|^2+6t)^2+16|p_3|^2.
\]
\end{rmk}
\begin{proof}
By iteration, it suffices to show there exist strategies of Player I or Player II such that
\begin{equation}\label{gauge iteration3}
\begin{aligned}
&((p_1+\sqrt{2}\varepsilon bv_1)^2+(p_2+\sqrt{2}\varepsilon bv_2)^2+j\varepsilon^2)^2\\
&+16(p_3+{1\over 2}\sqrt{2}\varepsilon b(p_1v_2-p_2v_1))^2
\leq (p_1^2+p_2^2+(j+6)\varepsilon^2)^2+16p_3^2.
\end{aligned}
\end{equation}
Indeed, the left hand side is calculated to be
\[
\begin{aligned}
&(p_1^2+p_2^2+(j+2)\varepsilon^2)^2+8\varepsilon^2(p_1^2+p_2^2)+16p_3^2\\
&+4\sqrt{2}\varepsilon b( (p_1^2+p_2^2+2\varepsilon^2)(p_1v_1+p_2v_2)+4(p_1p_3v_2-p_2p_3v_1))
\end{aligned}
\]
As in the proof of Lemma \ref{traj upper bound}, either Player I or Player II may let 
\[
b( (p_1^2+p_2^2+2\varepsilon^2)(p_1v_1+p_2v_2)+4(p_1p_3v_2-p_2p_3v_1))\leq 0
\]
with no regard for their opponents strategies. Hence, by a strategy of either Player I or Player II, we have
\[
\begin{aligned}
((p_1+\sqrt{2}\varepsilon bv_1)^2+&(p_2+\sqrt{2}\varepsilon bv_2)^2+j\varepsilon^2)^2+16(p_3+{1\over 2}\sqrt{2}\varepsilon b(p_1v_2-p_2v_1))^2\\
\leq &(p_1^2+p_2^2+(j+2)\varepsilon^2)^2+8\varepsilon^2(p_1^2+p_2^2)+16p_3^2\\
\leq &(p_1^2+p_2^2+(j+6)\varepsilon^2)^2+16p_3^2,
\end{aligned}
\]
which proves (\ref{gauge iteration3}).
\end{proof}
\begin{rmk}\label{rmk upper bound}
For any $\hat{p}\in \mathcal{H}$, $c\in \mathbb{R}$ and $L>0$, let 
\begin{equation}\label{gauge tran}
\hat{G}(p)=c+LG(\hat{p}^{-1}\cdot p).
\end{equation}
Our proof above can be directly generalized to show that
\[
\hat{G}(\zeta^N)\leq c+L(|p_1-\hat{p}_1|^2+|p_2-\hat{p}_2|^2+6N\varepsilon^2)^2+16L|p_3-\hat{p}_3+{1\over 2}(p_1\hat{p}_2-p_2\hat{p}_1)|^2
\]
with either a strategy of Player I or a strategy of Player II. 
\end{rmk}

We now return to the proof of Theorem \ref{game approximation}, which actually rests on showing that $\ol{u}$ and $\ul{u}$, as defined in (\ref{upper limit}) and (\ref{lower limit}), are respectively a subsolution and a supersolution of (MCF). (Note that our definitions are valid since the game value $u_\varepsilon$ are bounded uniformly for all $\varepsilon>0$ by its definition.) Moreover, we show that $\ol{u}(p, 0)\leq \ul{u}(p, 0)$ and $\ol{u}$ and $\ul{u}$ are constant outside a compact set. Then it follows immediately from the comparison principle (Theorem \ref{comparison theorem}) that $\ol{u}\leq \ul{u}$ and therefore $u^\varepsilon\to u$ locally uniformly as $\varepsilon\to 0$. 

\begin{prop}[Constant value outside a compact set]\label{prop const outside}
Assume that $u_0$ is uniformly continuous function in
$\mathcal{H}$ and is a constant $C\in \mathbb{R}$ outside a compact set. Let $u^\varepsilon$ be the
value function defined by (\ref{eq:valfun1}). Then for any $T>0$, $\ol{u}(p, t)=\ul{u}(p, t)=C$ for all $p\in \mathcal{H}$ outside a compact set and for all $t\in [0, T]$.
\end{prop}
\begin{proof}
Suppose there exists $B_r$ such that $u_0(p)=C$ for any $p\in \mathcal{H}\setminus B_r$.
Then for any $\hat{p}\in \mathcal{H}\setminus B_r$ and $t\geq 0$, we use the strategy of Player I introduced in Lemma \ref{traj lower bound}, we get $\zeta^N\in \mathcal{H}\setminus B_r$ regardless of Player II's choices, which implies that
\[
u^\varepsilon(p, t)\leq u_0(\zeta^N)=C.
\]
Similarly, we may use the strategy of Player II to deduce that
\[
u^\varepsilon(p, t)\geq C.
\]
Hence, $u^\varepsilon= C$ and $\ol{u}=\ul{u}=C$ in $\mathcal{H}\setminus B_r$.
\end{proof}

\begin{prop}\label{prop initial data}
Assume that $u_0$ is uniformly continuous function in $\mathcal{H}$ and is constant outside a compact set. Let $u^\varepsilon$ be the value function defined by (\ref{eq:valfun1}). Then $\ol{u}(p, 0)\leq u_0(p)$ and $\ul{u}(p, 0)\geq u_0(p)$ for all $p\in \mathcal{H}$.
\end{prop}
In order to prove this result, we first need to regularize the initial data with the smooth gauge $G$ in (\ref{gauge}). We define 
\begin{equation}\label{sup-convolution}
\psi^L(p)=\sup_{q\in \mathcal{H}}\{u_0(q)-LG(p^{-1}\cdot q)\} 
\end{equation}
and 
\begin{equation}\label{inf-convolution}
\psi_L(p)=\inf_{q\in \mathcal{H}}\{u_0(q)+LG(p^{-1}\cdot q)\},
\end{equation}
for any $p\in \mathcal{H}$ and fixed $L>0$.
These two functions are called the \textit{sup-convolution} and \textit{inf-convolution} of $u_0$ respectively. Our definitions here are slightly different from those in \cite{W} in that we plug $p^{-1}\cdot q$ instead of $q\cdot p^{-1}$ in $G$. However, the properties remain the same. We present one of the important properties for our use.
\begin{lem}[Approximation by semi-convolutions]\label{lem conv}
Assume that $u_0$ is uniformly continuous on $\mathcal{H}$ and is constant outside a compact set. Let $\psi^L$ and $\psi_L$ be respectively defined as in (\ref{sup-convolution}) and (\ref{inf-convolution}). Then $\psi^L$ and $\psi_L$ converge to $u_0$ uniformly in $\mathcal{H}$ as $L\to \infty$.
\end{lem}
\begin{proof}
We only show the statement for $\psi^L$.  The proof for the statement on $\psi_L$ is symmetric.

It is easily seen that 
\begin{equation}\label{conv1}
\psi^L\geq u_0 \text{ in $\mathcal{H}$.} 
\end{equation}
On the other hand, since $u_0$ is uniformly continuous, for any $p\in \mathcal{H}$, we may find $q_L\in \mathcal{H}$ such that
\[
\psi^L(p)=\sup_{q\in \mathcal{H}}\{u_0(p)-LG(p^{-1}\cdot q)\}=u_0(q_L)-LG(p^{-1}\cdot q_L).
\] 
By (\ref{conv1}), we have
\begin{equation}\label{conv2}
LG(p^{-1}\cdot q_L)\leq u_0(q_L)-u_0(p),
\end{equation}
which, by the boundedness of $u_0$, implies that
\[
|p^{-1}\cdot q_L|\leq (2K_0/L)^{1/4},
\]
where $K_0=\sup_{H}|u_0|$. By the uniform continuity of $u_0$, for any $\delta>0$, there exists $\varepsilon>0$ such that 
$|u_0(p)-u_0(q)|\leq \delta$ for any $p, q\in \mathcal{H}$ satisfying $|p^{-1}\cdot q|\leq \varepsilon$. Then we may let $L>0$ be sufficiently large such that $(2K_0/L)^{1/4}\leq \varepsilon$ and therefore
\[
u_0(q_L)-u_0(p)\leq \delta,
\]
which, combined with (\ref{conv1}), yields 
\[
|\psi^L(p)-u_0(p)|\leq \delta \text{ for all $p\in \mathcal{H}$}.
\]
\end{proof}

\begin{proof}[Proof of Proposition \ref{prop initial data}]
We arbitrarily fix $\hat{p}\in \mathcal{H}$. By Lemma \ref{lem conv}, for any $\delta>0$, there exists $L>0$  such that
\[
\psi^L(\hat{p})\leq u_0(\hat{p})+\delta,
\]
which implies that
\[
u_0(p)\leq u_0(\hat{p})+\delta+LG(\hat{p}^{-1}\cdot p).
\]
Let us use the right hand side, which is exactly $\hat{G}$ in (\ref{gauge tran}) with $c=u_0(\hat{p})+\delta$, as the objective function of the games. Suppose the game value is $w^\varepsilon$. Then by using the special strategy of Player I given in Lemma \ref{traj upper bound} and Remark \ref{rmk upper bound}, we obtain a game estimate
\[
\begin{aligned}
w^\varepsilon(p, t)\leq &u_0(\hat{p})+\delta+LG(\hat{p}^{-1}\cdot \zeta^N)\\
\leq& u_0(\hat{p})+\delta+L(|p_1-\hat{p}_1|^2+|p_2-\hat{p}_2|^2+6N\varepsilon^2)^2\\
 &+16L|p_3-\hat{p}_3+{1\over 2}(p_1\hat{p}_2-p_2\hat{p}_1)|^2
\end{aligned}
\]
no matter what choices are made by Player II during the game.
On the other hand, since it is clear that $u^\varepsilon\leq w^\varepsilon$ and $N\varepsilon^2\leq t$, we get
\[
\begin{aligned}
u^\varepsilon(p, t)\leq u_0(\hat{p})+\delta+L(|p_1-\hat{p}_1|^2+&|p_2-\hat{p}_2|^2+6t)^2\\
&+16L|p_3-\hat{p}_3+{1\over 2}(p_1\hat{p}_2-p_2\hat{p}_1)|^2.
\end{aligned}
\]
Taking the relaxed limit of $u^\varepsilon$ at $(\hat{p}, 0)$ as $\varepsilon\to 0$, we have
\[
\ol{u}(\hat{p}, 0)\leq u_0(\hat{p})+\delta.
\]
We finally send $\delta\to 0$ and get $\ol{u}(\hat{p}, 0)\leq u_0(\hat{p})$ for any $\hat{p}\in \mathcal{H}$.

The proof for the statement that $\ul{u}(p, 0)\geq u_0(p)$ for all $p\in \mathcal{H}$ is symmetric. In fact, the key is to use the strategy of Player II introduced in Lemma \ref{traj upper bound} and Remark \ref{rmk upper bound} to deduce
\[
\begin{aligned}
u^\varepsilon(p, t)\geq u_0(\hat{p})-\delta-L(|p_1-\hat{p}_1|^2+&|p_2-\hat{p}_2|^2+6t)^2\\
&-16L|p_3-\hat{p}_3+{1\over 2}(p_1\hat{p}_2-p_2\hat{p}_1)|^2.
\end{aligned}
\]

\end{proof}

\begin{prop}[Axial symmetry of the game values]\label{prop axisym}
Suppose that $u_0$ is uniformly continuous on $\mathcal{H}$ and is spatially axisymmetric with respect to the vertical axis. Let $u^\varepsilon$ be the value function defined as in (\ref{eq:valfun1}) Then $u^\varepsilon$, $\ol{u}$ and $\ul{u}$ are also spatially axisymmetric about the vertical axis.
\end{prop}
\begin{proof}
We argue by induction. Assume that $u^\varepsilon(p, t)=u^\varepsilon(p', t)$ for some $t\geq 0$ and for any $p, p'\in \mathcal{H}$ such that
\begin{equation}\label{axisym}
p_1^2+p_2^2=(p_1')^2+(p_2')^2 \text{ and }p_3=p_3'.
\end{equation}
We aim to show $u^\varepsilon(p, t+\varepsilon^2)=u^\varepsilon(p', t+\varepsilon^2)$ for all $p, p'\in \mathcal{H}$ satisfying the condition (\ref{axisym}).

Since the dynamic programming principle (\ref{dpp}) gives 
\[
u^\varepsilon(p, t+\varepsilon^2)=\min_{v\in S^1_h}\max_{b=\pm1} u^\varepsilon\left(p\cdot (\sqrt{2}\varepsilon bv), t\right),
\]
there exists $v\in S^1_h$ such that 
\begin{equation}\label{dpp ext}
u^\varepsilon(p, t+\varepsilon^2)= \max_{b=\pm 1}u^\varepsilon\left(p\cdot (\sqrt{2}\varepsilon bv), t\right).
\end{equation}
We claim that there is $v'\in S^1_h$ such that the coordinates of $p\cdot (\sqrt{2}\varepsilon bv)$ and $p'\cdot (\sqrt{2}\varepsilon bv')$ satisfy (\ref{axisym}) as well. Indeed, as
\[
p\cdot (\sqrt{2}\varepsilon bv)=\left(p_1+\sqrt{2}\varepsilon bv_1, p_2+\sqrt{2}\varepsilon bv_2, p_3+{1\over 2}\sqrt{2}\varepsilon b(p_1v_2-p_2v_1)\right)
\]
and 
\[
p'\cdot (\sqrt{2}\varepsilon bv')=\left(p'_1+\sqrt{2}\varepsilon bv'_1, p'_2+\sqrt{2}\varepsilon bv'_2, p'_3+{1\over 2}\sqrt{2}\varepsilon b(p'_1v'_2-p'_2v'_1)\right),
\] 
we are looking for $v'_1, v'_2\in S^1_h$ such that
\[
\left\{
\begin{aligned}
&(p'_1+\sqrt{2}\varepsilon bv'_1)^2+(p'_2+\sqrt{2}\varepsilon bv'_2)^2=(p_1+\sqrt{2}\varepsilon bv_1)^2+(p_2+\sqrt{2}\varepsilon bv_2)^2\\
&p'_3+{1\over 2}\sqrt{2}\varepsilon b(p'_1v'_2-p'_2v'_1)=p_3+{1\over 2}\sqrt{2}\varepsilon b (p_1v_2-p_2v_1).
\end{aligned}
\right.
\] 
Since $p$ and $p'$ satisfy (\ref{axisym}), it suffices to solve the linear system
\[
\left\{
\begin{aligned}
&p'_1v'_1+p'_2v'_2=p_1v_1+p_2v_2,\\
&-p'_2 v'_1+p'_1v'_2=-p_2v_1+p_1v_2.
\end{aligned}
\right.
\]
The problem is trivial if $p_1^2+p_2^2=(p'_1)^2+(p'_2)^2=0$. When $p_1^2+p_2^2=(p'_1)^2+(p'_2)^2\neq 0$, we get a unique pair of solutions
\[
\begin{aligned}
&v'_1={1\over {(p'_1)^2+(p'_2)^2}}\left((p_1p'_1+p_2p'_2)v_1+(p'_1p_2-p_1p'_2)v_2\right),\\
&v'_2={1\over {(p'_1)^2+(p'_2)^2}}\left((p_1p'_2-p'_1p_2)v_1+(p_1p'_1+p_2p'_2)v_2\right).
\end{aligned}
\]
Thanks to the relation (\ref{axisym}), it is easy to verify that $v'=(v'_1, v'_2, 0)\in S^1_h$, i.e., $(v'_1)^2+(v'_2)^2=1$. 
We complete the proof of the claim.

In view of the induction hypothesis, we obtain 
\[
u^\varepsilon\left(p'\cdot (\sqrt{2}\varepsilon bv'), t\right)=u^\varepsilon\left(p\cdot(\sqrt{2}\varepsilon bv), t\right) \text{ for both $b=\pm 1$,}
\] 
which, together with the dynamic programming (\ref{dpp}) and (\ref{dpp ext}), yields 
\[
u^\varepsilon(p', t+\varepsilon^2)\leq \max_{b=\pm 1}u^\varepsilon\left(p'\cdot (\sqrt{2}\varepsilon bv'), t\right)\leq u^\varepsilon(p, t+\varepsilon^2).
\]
We may similarly prove that $u^\varepsilon(p', t+\varepsilon^2)\geq u^\varepsilon(p, t+\varepsilon^2)$ and therefore $u^\varepsilon(p', t+\varepsilon^2)= u^\varepsilon(p, t+\varepsilon^2)$ for all $p, p'\in \mathcal{H}$ satisfying (\ref{axisym}). 

It follows from the definitions (\ref{upper limit})--(\ref{lower limit}) of half relaxed limits that the same results for $\ol{u}$ and $\ul{u}$ hold.

\end{proof}

\begin{prop}\label{game subsolution}
Assume that $u^\varepsilon$ satisfies the dynamic programming principle (\ref{dpp}).
Let $\ol{u}$ be the upper relaxed limit defined as in (\ref{upper limit}). Then $\ol{u}$ is a subsolution of (\ref{mcf1}).
\end{prop}
\begin{proof}
Assume that there exists $(\hat{p}, \hat{t})\in \mathcal{H}\times (0, \infty)$ and $\phi\in C^2(\mathcal{H}\times (0, \infty))$ such that $\ol{u}-\phi$ attains a strict maximum at $(\hat{p}, \hat{t})$. Then by definitions of $\ol{u}$, we may take a sequence, still indexed by $\varepsilon$, $(p^\varepsilon, t^\varepsilon)\in \mathcal{H}\times (0, \infty)$ such that $(p^\varepsilon, t^\varepsilon) \to (\hat{p}, \hat{t})$ and 
$u^\varepsilon(p^\varepsilon, t^\varepsilon)\to \ol{u}(\hat{p}, \hat{t})$ as $\varepsilon\to 0$ and 
\begin{equation}\label{max convergence}
u^\varepsilon(p^\varepsilon, t^\varepsilon)-\phi(p^\varepsilon, t^\varepsilon)= \max_{B_r(\hat{p}, \hat{t})}(u^\varepsilon-\phi)
\end{equation}
Applying the dynamic programming principle (\ref{dpp}) with $(p, t)=(p^\varepsilon, t^\varepsilon)$, we have
\[
u^\varepsilon(p^\varepsilon, t^\varepsilon)=\min_{v}\max_{b} u^\varepsilon\left(p^\varepsilon\cdot (\sqrt{2}\varepsilon bv), t^\varepsilon-\varepsilon^2\right),
\]
which, combined with (\ref{max convergence}), implies that
\[
\phi(p^\varepsilon, t^\varepsilon)\leq \min_{v}\max_{b}\phi\left(p^\varepsilon\cdot (\sqrt{2}\varepsilon bv), t^\varepsilon-\varepsilon^2\right).
\]
We next use the Taylor expansion for the right hand side at $(p^\varepsilon, t^\varepsilon)$ and obtain
\begin{equation}\label{taylor1}
\varepsilon^2 \phi_t(p^\varepsilon, t^\varepsilon)-\min_v\max_b(\langle \sqrt{2}\varepsilon bv, \nabla \phi(p^\varepsilon, t^\varepsilon)\rangle +\varepsilon^2\langle (\nabla_H^2 \phi)^\ast(p^\varepsilon, t^\varepsilon) v_h, v_h\rangle)\leq o(\varepsilon^2),
\end{equation}
where $v_h$ is the horizontal projection of $v$, i.e., $v_h=(v^1, v^2)$ for any $v=(v^1, v^2, v^3)$.
Since $v=(v^1, v^2, 0)$, we may rewrite (\ref{taylor1}) as
\begin{equation}\label{taylor2}
\varepsilon^2 \phi_t(p^\varepsilon, t^\varepsilon)-\min_v\max_b(\langle \sqrt{2}\varepsilon bv_h, \nabla_H \phi(p^\varepsilon, t^\varepsilon)\rangle +\varepsilon^2\langle (\nabla_H^2 \phi)^\ast(p^\varepsilon, t^\varepsilon) v_h, v_h\rangle)\leq o(\varepsilon^2),
\end{equation}

We discuss two cases:

\noindent\textbf{Case A:} $\nabla_H\phi(\hat{p}, \hat{t})\neq 0$. Then $\nabla_H\phi(p^\varepsilon, t^\varepsilon)\neq 0$ for all sufficiently small $\varepsilon>0$. Letting 
\[
\tilde{v}=\frac{1}{|\nabla_H \phi(p^\varepsilon, t^\varepsilon)|}(X_2\phi(p^\varepsilon, t^\varepsilon), -X_1\phi(p^\varepsilon, t^\varepsilon), 0)
\] 
with 
\[
\tilde{v}_h=\frac{1}{|\nabla_H\phi(p^\varepsilon, t^\varepsilon)|}(X_2\phi(p^\varepsilon, t^\varepsilon), -X_1\phi(p^\varepsilon, t^\varepsilon)),
\] 
we have from (\ref{taylor1})
\begin{equation}\label{sub derivation1}
\phi_t(p^\varepsilon, t^\varepsilon)-\langle (\nabla_H^2 \phi)^\ast(p^\varepsilon, t^\varepsilon) \tilde{v}_h, \tilde{v}_h\rangle)\leq o(1) .
\end{equation}
Noticing that 
\[
\tilde{v}_h\otimes \tilde{v}_h=I-\frac{\nabla_H\phi(p^\varepsilon, t^\varepsilon)\otimes \nabla_H\phi(p^\varepsilon, t^\varepsilon)}{|\nabla_H\phi(p^\varepsilon, t^\varepsilon)|^2},
\]
we are thus led from (\ref{sub derivation1}) to 
\begin{equation}\label{sub derivation2}
\phi_t(p^\varepsilon, t^\varepsilon)-\tr\left[\bigg(I-\frac{\nabla_H\phi(p^\varepsilon, t^\varepsilon)\otimes \nabla_H\phi(p^\varepsilon, t^\varepsilon)}{|\nabla_H\phi(p^\varepsilon, t^\varepsilon)|^2}\bigg)(\nabla_H^2 \phi)^\ast(p^\varepsilon, t^\varepsilon)\right]\leq o(1).
\end{equation}
Sending $\varepsilon\to 0$, we get
\[
\phi_t(\hat{p}, \hat{t})-\tr\left[\bigg(I-\frac{\nabla_H\phi(\hat{p}, \hat{t})\otimes \nabla_H\phi(\hat{p}, \hat{t})}{|\nabla_H\phi(\hat{p}, \hat{t})|^2}\bigg)(\nabla_H^2 \phi)^\ast(\hat{p}, \hat{t})\right]\leq 0.
\]

\noindent\textbf{Case B:} $\nabla_H\phi(\hat{p}, \hat{t})= 0$. In this case, we have by Definition \ref{trad defi} that 
\begin{equation}\label{sub derivation3}
\nabla_H^2\phi^\ast(\hat{p}, \hat{t})=0.
\end{equation}
If $\nabla_H\phi(p^\varepsilon, t^\varepsilon)\neq 0$ for all $\varepsilon>0$, we may follow the same argument as in Case A, passing to the limit for (\ref{sub derivation2}) as $\varepsilon\to 0$ with an application of (\ref{sub derivation3}), and get
\begin{equation}\label{sub derivation4}
\phi_t(\hat{p}, \hat{t})\leq 0,
\end{equation}
as desired. 

If there exists a subsequence $\varepsilon_j$ such that $\nabla_H\phi(p^{\varepsilon_j}, t^{\varepsilon_j})=0$ for all $j$, then it follows from (\ref{taylor1}) that 
\[
\phi_t(p^{\varepsilon_j}, t^{\varepsilon_j})-\langle (\nabla_H^2 \phi)^\ast(p^{\varepsilon_j}, t^{\varepsilon_j}) v_h, v_h\rangle)\leq o(1) \text{ for some $v$,}
\] 
which again implies (\ref{sub derivation4}) as the limit when $\varepsilon_j\to 0$.
\end{proof}

\begin{prop}\label{game supersolution}
Assume that $u^\varepsilon$ satisfies the dynamic programming principle (\ref{dpp}). 
Let $\ul{u}$ be the lower relaxed limit defined as in (\ref{lower limit}). Then $\ul{u}$ is a supersolution of (\ref{mcf1}).
\end{prop}

In order to facilitate the proof, let us present an elementary result. 

\begin{lem}[Lemma 4.1 in \cite{GL}]\label{elementary lem}
Suppose $\xi$ is a unit vector in $\mathbb{R}^2$ and $X$ is a real symmetric $2\times 2$
  matrix, then there exists a constant $M>0$ that depends only on the norm of $X$, such that for any unit
  vector $v\in \mathbb{R}^2$,
  \begin{equation}
    |\langle X\xi^\bot, \xi^\bot\rangle-\langle X v, v\rangle|\leq
    M|\langle\xi, v\rangle|,
  \end{equation}
  where $\xi^\bot$ denotes a unit orthonormal vector of $\xi$.
\end{lem}
\begin{proof}
Let $\cos\theta=\langle \xi^\bot, v\rangle$ and $\sin\theta=\langle \xi, v\rangle$. Then we have 
\[
\begin{aligned}
&|\langle X\xi^\bot, \xi^\bot\rangle -\langle Xv, v\rangle|\\
=& |\tr\left(X(\xi^\bot\otimes \xi^\bot- v\otimes v)\right)|\\
\leq &\|X\| \|\xi^\bot\otimes \xi^\bot-(\xi\sin\theta+\xi^\bot \cos\theta)\otimes (\xi\sin\theta+\xi^\bot\cos\theta)\|\\
=&\|X\|\|\sin^2\theta \xi^\bot\otimes \xi^\bot-\sin\theta\cos\theta(\xi\bot\xi^\bot+\xi^\bot\otimes \xi)\|\\
\leq & M|\sin\theta|,
\end{aligned}
\]
where $M>0$ depends on $\|X\|$. 
\end{proof}
We refer the reader to \cite[Lemma 2.3]{L1} for a higher dimensional extension of this lemma. 

\begin{proof}[Proof of Proposition \ref{game supersolution}]
Assume that there exists $(\hat{p}, \hat{t})\in \mathcal{H}\times (0, \infty)$ and $\phi\in C^2(\mathcal{H}\times (0, \infty))$ such that $\ul{u}-\phi$ attains a strict minimum at $(\hat{p}, \hat{t})$. We may again take a sequence $(p^\varepsilon, t^\varepsilon)\in \mathcal{H}\times (0, \infty)$ such that $(p^\varepsilon, t^\varepsilon) \to (\hat{p}, \hat{t})$ and 
$u^\varepsilon(p^\varepsilon, t^\varepsilon)\to \ul{u}(\hat{p}, \hat{t})$ as $\varepsilon\to 0$ and 
\begin{equation}\label{min convergence}
u^\varepsilon(p^\varepsilon, t^\varepsilon)-\phi(p^\varepsilon, t^\varepsilon)= \min_{B_r(\hat{p}, \hat{t})}(u^\varepsilon-\phi)
\end{equation}
Applying the dynamic programming principle (\ref{dpp}) with $(p, t)=(p^\varepsilon, t^\varepsilon)$, we have
\[
u^\varepsilon(p^\varepsilon, t^\varepsilon)=\min_{v}\max_{b} u^\varepsilon(p^\varepsilon\cdot \sqrt{2}\varepsilon bv, t^\varepsilon-\varepsilon^2).
\]
It then follows from (\ref{min convergence}) that
\[
\phi(p^\varepsilon, t^\varepsilon)\geq \min_{v}\max_{b}\phi(p^\varepsilon\cdot \sqrt{2}\varepsilon bv, t^\varepsilon-\varepsilon^2).
\]
As an analogue of (\ref{taylor2}), the Taylor expansion at $(p^\varepsilon, t^\varepsilon)$ yields
\begin{equation}\label{taylor3}
\phi_t(p^\varepsilon, t^\varepsilon)-\min_v\max_b({1\over \varepsilon}\langle \sqrt{2} bv_h, \nabla_H \phi(p^\varepsilon, t^\varepsilon)\rangle +\langle (\nabla_H^2 \phi)^\ast(p^\varepsilon, t^\varepsilon) v_h, v_h\rangle)\geq o(1),
\end{equation}
as $\varepsilon\to 0$.

We again divide our discussion into two situations.

\noindent \textbf{Case A:} $\nabla_H\phi(\hat{p}, \hat{t})\neq 0$. Then $\nabla\phi(p^\varepsilon, t^\varepsilon)\neq 0$ for all sufficiently small $\varepsilon>0$. We adopt Lemma \ref{elementary lem} and get
\[
\begin{aligned}
&\max_b\left({1\over \varepsilon}\langle \sqrt{2} bv_h, \nabla_H \phi(p^\varepsilon, t^\varepsilon)\rangle +\langle (\nabla_H^2 \phi)^\ast(p^\varepsilon, t^\varepsilon) v_h, v_h\rangle\right)\\
\leq &\langle (\nabla_H^2 \phi)^\ast(p^\varepsilon, t^\varepsilon) \tilde{v}_h, \tilde{v}_h\rangle+\left(M+{\sqrt{2}\over\varepsilon}\right)|\langle v_h, \phi_H(p^\varepsilon, t^\varepsilon)\rangle|,
\end{aligned}
\]
where 
\[
\tilde{v}_h=\frac{1}{|\nabla_H\phi(p^\varepsilon, t^\varepsilon)|}(X_2\phi(p^\varepsilon, t^\varepsilon), -X_1\phi(p^\varepsilon, t^\varepsilon)),
\] 
as given in the proof of Proposition \ref{game subsolution}. It is now clear, by taking $v_h=\tilde{v}_h$, that
\[
\begin{aligned}
&\min_v\max_b\left({1\over \varepsilon}\langle \sqrt{2} bv_h, \nabla_H \phi(p^\varepsilon, t^\varepsilon)\rangle +\langle (\nabla_H^2 \phi)^\ast(p^\varepsilon, t^\varepsilon) v_h, v_h\rangle\right)\\
\leq &\langle (\nabla_H^2 \phi)^\ast(p^\varepsilon, t^\varepsilon) \tilde{v}_h, \tilde{v}_h\rangle,
\end{aligned}
\]
which implies through (\ref{taylor3}) that
\[
\phi_t(p^\varepsilon, t^\varepsilon)-\langle (\nabla_H^2 \phi)^\ast(p^\varepsilon, t^\varepsilon) \tilde{v}_h, \tilde{v}_h\rangle\geq o(1).
\]
Letting $\varepsilon\to 0$, we obtain 
\[
\phi_t(\hat{p}, \hat{t})-\tr\left[\bigg(I-\frac{\nabla_H\phi(\hat{p}, \hat{t})\otimes \nabla_H\phi(\hat{p}, \hat{t})}{|\nabla_H\phi(\hat{p}, \hat{t})|^2}\bigg)(\nabla_H^2 \phi)^\ast(\hat{p}, \hat{t})\right]\geq 0.
\]

\noindent\textbf{Case B:} $\nabla_H\phi(\hat{p}, \hat{t})= 0$. We may further assume (\ref{sub derivation3}) again in this case.
We may apply the same argument above and get 
\begin{equation}\label{sub derivation5}
\phi_t(\hat{p}, \hat{t})\geq 0,
\end{equation}
provided that $\nabla_H\phi(p^\varepsilon, t^\varepsilon)\neq 0$ for all $\varepsilon>0$. It remains to show (\ref{sub derivation5}) when there is a subsequence $\varepsilon_j$ such that $\nabla_H\phi(p^{\varepsilon_j}, t^{\varepsilon_j})=0$. 
By (\ref{taylor3}), we have on this occasion
\[
\phi_t(p^{\varepsilon_j}, t^{\varepsilon_j})-\langle (\nabla_H^2 \phi)^\ast(p^{\varepsilon_j}, t^{\varepsilon_j}) v_h, v_h\rangle)\geq o(1) \text{ for some $v$.}
\]
Sending $\varepsilon\to 0$, we get (\ref{sub derivation5}).
\end{proof}

We are now in a position to prove Theorem \ref{game approximation}.
\begin{proof}[Proof of Theorem \ref{game approximation}]
In terms of Proposition \ref{prop axisym}, Proposition \ref{game subsolution} and Proposition \ref{game supersolution}, $\ol{u}$ and $\ul{u}$ are respectively a subsolution and a supersolution of (\ref{mcf1}) that are axisymmetric with respect to the vertical axis. For any $T>0$, $\ol{u}(p, t)$ and $\ul{u}(p, t)$ are constant outside a compact set of $\mathcal{H}$ for all $t\in [0, T]$, owing to Proposition \ref{prop const outside}. Also, since $\ol{u}(p, 0)\leq u_0(p)$ and $\ul{u}(p, 0)\geq u_0(p)$ for all $p\in \mathcal{H}$, we may apply Theorem \ref{comparison theorem} to get $\ol{u}\leq \ul{u}$ in $\mathcal{H}\times [0, T]$. As it is obvious that $\ol{u}\geq \ul{u}$, we get $\ol{u}=\ul{u}$ in $\mathcal{H}\times [0, T]$
with $u(\cdot, 0)=u_0(\cdot)$. In conclusion, $u=\ol{u}=\ul{u}$ is the unique continuous solution of (MCF) and the locally uniform convergence $u^\varepsilon\to u$ follows immediately.
\end{proof}

\section{Stability}\label{sec stability}
The following stability result is standard in the theory of viscosity solutions. 
\begin{thm}[Stability under the uniform convergence]\label{thm stability}
Let $u^\varepsilon$ be solutions of (\ref{mcf1}) and $u^\varepsilon\to u$ locally uniformly in $\mathcal{H}\times [0, \infty)$. Then $u$ is also a solution of (\ref{mcf1}).
\end{thm}

\begin{lem}\label{lem stability}
If $u^\varepsilon$ is a subsolution (resp., supersoution) of (\ref{mcf1}) for all small $\varepsilon>0$, then 
\[
\ol{u}=\limsups_{\varepsilon\to \infty}u^\varepsilon\quad
(\text{resp., } \ul{u}=\liminfs_{\varepsilon\to \infty} u^\varepsilon) 
\]
is also a subsolution (resp., supersoution) of (\ref{mcf1}).
\end{lem}
\begin{proof}
Suppose there exists $\phi\in C^2(\mathcal{H}\times [0, \infty))$ and $(\hat{p}, \hat{t})\in \mathcal{H}\times (0, \infty)$ such that $\ol{u}-\phi$ attains a strict maximum at $(\hat{p}, \hat{t})$. Then by the convergence of maximizers as shown in \cite[Lemma 2.2.5]{G1}, we can take subsequences of $p^\varepsilon, t^\varepsilon$ and $u^\varepsilon$, still indexed by $\varepsilon$, satisfying $(p^\varepsilon, t^\varepsilon)\to (\hat{p}, \hat{t})$ as $\varepsilon\to 0$ and 
\[
(u^\varepsilon-\phi)(p^\epsilon, t^\epsilon)=\max_{\mathcal{H}\times [0, \infty)}( u^\varepsilon-\phi).
\]
We discuss two cases.\\
\textbf{Case 1:} $\nabla_H\phi(\hat{p}, \hat{t})\neq 0$. Then $\nabla_H\phi(p^\varepsilon, t^\varepsilon)\neq 0$ for all $\varepsilon>0$ small. We apply Definition \ref{trad defi} and get
\[
\phi_t-\tr\left[\bigg(I-\frac{\nabla_H\phi\otimes \nabla_H\phi}{|\nabla_H\phi|^2}\bigg)(\nabla_H^2 \phi)^\ast\right]\leq 0 \text{ at $(p^\varepsilon, t^\varepsilon)$}.
\]
Sending $\varepsilon\to 0$, we get the desired inequality
\[
\phi_t-\tr\left[\bigg(I-\frac{\nabla_H\phi\otimes \nabla_H\phi}{|\nabla_H\phi|^2}\bigg)(\nabla_H^2 \phi)^\ast\right]\leq 0 \text{ at $(\hat{p}, \hat{t})$}.
\]
\textbf{Case 2:} $\nabla_H\phi(\hat{p}, \hat{t})=0$. Then by Definition \ref{trad defi} we only need to discuss the situation when $(\nabla_H^2 \phi)^\ast(\hat{p}, \hat{t})=0$ also holds. Hence, $\nabla_H \phi(p^\varepsilon, t^\varepsilon)\to 0$ and $(\nabla_H^2 \phi)^\ast(p^\varepsilon, t^\varepsilon)\to 0$ as $\varepsilon\to 0$. If there exists a subsequence $\varepsilon_j$ such that $\nabla_H \phi(p^{\varepsilon_j}, t^{\varepsilon_j})\neq 0$, then we have
\[
\phi_t-\tr\left[\bigg(I-\frac{\nabla_H\phi\otimes \nabla_H\phi}{|\nabla_H\phi|^2}\bigg)(\nabla_H^2 \phi)^\ast\right]\leq 0 \text{ at $(p^{\varepsilon_j}, t^{\varepsilon_j})$}.
\]
Passing to the limit $j\to \infty$, we obtain $\phi_t(\hat{p}, \hat{t})\leq 0$.

If, on the other hand, $\nabla_H \phi(p^\varepsilon, t^\varepsilon)=0$ for all $\varepsilon>0$ small, then we get $\phi_t(p^\varepsilon, t^\varepsilon)\leq 0$ and the limit immediately yields $\phi_t(\hat{p}, \hat{t})\leq 0$, which completes our proof.

One may similarly prove that $\ul{u}=\liminfs_{\varepsilon\to 0} u^\varepsilon$ is a supersolution provided that $u^\varepsilon$ is a supersolution for all $\varepsilon>0$ small.
\end{proof}

\begin{proof}[Proof of Theorem \ref{thm stability}]
Let
\[
\ol{u}=\limsups_{\varepsilon\to 0} u^\varepsilon \text{ and }\ul{u}=\liminfs_{\varepsilon\to 0} u^\varepsilon.
\]
Then in virtue of Lemma \ref{lem stability}, $\ol{u}$ is a subsolution of (\ref{mcf1}) and $\ul{u}$ is a supersolution of (\ref{mcf1}). Noting that $u^\varepsilon\to u$ locally uniformly, we must have $u=\ol{u}=\ul{u}$ and therefore $u$ is a solution of (\ref{mcf1}).
\end{proof}

\section{Properties of the evolution}\label{sec properties}
We have shown that there is a unique solution $u$ of (MCF) for any given continuous function $u_0$ which is axisymmetric with respect to the vertical axis and attains constant value outside a compact set. Let us turn to discuss the surface evolution described by the level-set equation (MCF). More precisely, given an axisymmetric compact surface $\Gamma_0\subset \mathcal{H}$, we choose $u_0\in C(\mathcal{H})$ such that it is axisymmetric constant outside a compact set and satisfies
\begin{equation}\label{initial surface}
\Gamma_0=\{p:\mathcal{H}: u_0(p)=0\}.
\end{equation}
We then solve (MCF) for the unique solution $u$ and get the surface 
\begin{equation}\label{surface evolution}
\Gamma_t=\{p\in \mathcal{H}: u(p, t)=0\} \text{ for any $t\geq 0$.}
\end{equation}

In what follows, we first show that the surface represented by the level-set $\Gamma_t$ of $u$ does not depend on the particular choice of $u_0$.

\subsection{Uniqueness of the surface evolution}
\begin{thm}[Invariance]\label{thm trans invariant}
Assume that $\theta:\mathbb{R}\to \mathbb{R}$ is continuous. If $u$ is a solution of (\ref{mcf1}). Then $w=\theta\circ u$ is also a solution of (\ref{mcf1}). 
\end{thm}
\begin{proof}
We prove the theorem in several steps.

\noindent\textbf{Step 1.} We first give the proof in the case that $\theta\in C^2(\mathbb{R})$ and $\theta'>0$. Suppose that there is $\phi\in C^2(\mathcal{H}\times [0, \infty))$ and $(\hat{p}, \hat{t})\in \mathcal{H}\times (0, \infty)$ such that $\theta\circ u-\phi$ attains a maximum at $(\hat{p}, \hat{t})$. Then it is clear that $u-h(\phi)$ attains a maximum at $(\hat{p}, \hat{t})$, where $h=\theta^{-1}\in C^2(\mathbb{R})$ with $h'>0$. Denote $\psi=h(\phi)$. Since $u$ is a subsolution of (\ref{mcf1}), we have
\[
\psi_t-\tr\left[\bigg(I-\frac{\nabla_H\psi\otimes \nabla_H\psi}{|\nabla_H\psi|^2}\bigg)(\nabla_H^2 \psi)^\ast\right]\leq 0 \text{ at $(\hat{p}, \hat{t})$}.
\]
Note that $\psi_t=h'\phi_t$, $\nabla_H\psi=h'\nabla_H\phi$ and 
\[
(\nabla_H^2\psi)^\ast=h''\nabla_H\phi\otimes \nabla_H\phi+h'(\nabla_H^2\phi)^\ast.
\]
It follows that 
\[
\phi_t-\tr\left[\bigg(I-\frac{\nabla_H\phi\otimes \nabla_H\phi}{|\nabla_H\phi|^2}\bigg)(\nabla_H^2 \phi)^\ast\right]\leq 0 \text{ at $(\hat{p}, \hat{t})$},
\]
which shows that $\theta\circ u$ is a subsolution of (\ref{mcf1}). An analogue of this argument yields that $\theta\circ u$ is also a supersolution. 

We also claim that $\theta\circ u$ remains being a solution when $\theta\in C^2(\mathbb{R})$ and $\theta'<0$. Indeed, when $\theta$ is a decreasing function, $-\theta$ is increasing. We obtain that $-\theta\circ u$ is a solution of (\ref{mcf1}). Thanks to the fact that the mean curvature flow is orientation-free or (\ref{mcf1}) is homogeneous in all of the derivatives, we easily see that $\theta\circ u$ is a solution as well. In particular, we note that $-u$ is a solution when $u$ is a solution.

\noindent \textbf{Step 2.} We generalize the consequence obtained in Step 1 for a continuous nondecreasing or nonincreasing function. Indeed, for any continuous nondecreasing function $\theta$, we may take $\theta_n\in C^2(\mathbb{R})$ with $\theta_n'>0$ for all $n=1, 2, \dots$ such that 
\[
\limsups_{n\to \infty}\theta_n\circ u= \theta\circ u.
\]
We refer the reader to \cite[Lemma 4.2.3]{G1} for details about the construction of $\theta_n$. Since $\theta_n\circ u$ is a solution of (\ref{mcf1}) for all $n$, as shown in Step 1, $\theta\circ u$ is a subsolution, due to Lemma \ref{lem stability}. 

To show that $\theta\circ u$ is a supersolution, we define $\tilde{\theta}(x)=\theta(-x)$ for any $x\in \mathbb{R}$ and observe that $\theta(u)=\tilde{\theta}(-u)$. Since $\tilde{\theta}$ is nonincreasing and $-u$ is a solution, we may apply a symmetric version of \cite[Lemma 4.2.3]{G1} to get $\theta(u)=\tilde{\theta}(-u)$ is a supersoluiton.

When $\theta$ is a continuous nonincreasing function, $-\theta$ is nondecreasing. We apply again the homogeneity of (\ref{mcf1}) to obtain that $\theta\circ u$ is a solution. Since the verification of definition of (\ref{mcf1}) is pointwise, one can further relax the monotonicity condition on $\theta$ to a local monotonicity condition. 

To conclude this step, we notice that $\max\{\min\{u, C\} -C\}$ is a solution for any $C>0$ provided that $u$ is a solution.

\noindent\textbf{Step 3.} We finally discuss the situation when $\theta$ is assumed to be continuous only. By Theorem \ref{thm stability}, it suffices to discuss the bounded function $\max\{\min\{u, C\} -C\}$ instead of $u$ for arbitrarily large $C>0$. We approximate $\theta$ uniformly by polynomials $\theta_m$ in $[-C-1, C+1]$. Since polynomials only have finitely many maximizers and minimizers, we may also assume each $\theta_m$ is constant near all of its local maximizers and minimizers. 

In fact, if, for instance, $\theta_m$ attains a local maximum at $x_0\in \mathbb{R}$, we take $\min\{\theta_m(x), \theta(x_0)-\varepsilon_m\}$, where $\varepsilon_m>0$ is sufficiently small ($\varepsilon_m\to 0$ as $m\to \infty$) such that $\theta_m$ is continuous. 

Now $\theta_m$ is locally nonincreasing or nondecreasing. We apply the result in Step 2 and find that $\theta_m\circ u$ is a solution of (\ref{mcf1}). Since $\theta_m\to \theta$ uniformly, by the stability result given in Theorem \ref{thm stability}, we see that $\theta\circ u$ is a solution by sending $m\to \infty$.
\end{proof}
An immediate consequence of the theorem above is that our generalized surface evolution does not depend on the choice of the initial level-set function $u_0$.
\begin{cor}[Independence of the choice of the initial function]\label{cor independence}
Suppose that $u_0$ and $\tilde{u}_0$ are continuous functions in $\mathcal{H}$ axisymmetric about the vertical axis and are constant outside a compact set $K\subset \mathcal{H}$. Let $\Gamma_0=\{p\in \mathcal{H}: u_0(p)=0\}=\{p\in \mathcal{H}: \tilde{u}_0(p)=0\}$ be bounded. Let $u$ and $\tilde{u}$ be the unique continuous solutions of (\ref{mcf1}) with the initial conditions $u_0$ and $\tilde{u}_0$ respectively. For any $t\geq 0$, set 
\[
\Gamma_t=\{p\in \mathcal{H}: u(p, t)=0\}\text{ and } \tilde{\Gamma}_t=\{p\in \mathcal{H}: \tilde{u}(p, t)=0\}.
\]
Then $\Gamma_t=\tilde{\Gamma}_t$ for all $t\geq 0$.
\end{cor}
\begin{proof}
We follow the proof of \cite[Theorem 5.1]{ES}. It is obvious, from Theorem \ref{comparison theorem} and Theorem \ref{game approximation}, that $u$ and $\tilde{u}$ are axisymmetric about the vertical axis.

We may assume $u_0\geq 0$ without changing the zero level set of $u_0$, since $|u|$ is a solution of (\ref{mcf1}) with the initial condition $u(p, 0)=|u_0|$ by Theorem \ref{thm trans invariant}. Similarly, let us also assume that $\tilde{u}_0\geq 0$ and $\tilde{u}\geq 0$.

For any $k = 1, 2, \dots$ let $E_0 = \emptyset$ and $E_k = \{p\in \mathcal{H}: u_0(p) > l/k\}$ such that
$E_k$ is nondecreasing and $\mathcal{H}\subset \Gamma_0=\cup_{k}E_k$.
Define
\[
a_k = \max_{\mathcal{H}\setminus E_{k-1}} \tilde{u}_0  (k = 1, 2 , . . . ). 
\]
Then we have $\lim_{k\to \infty} a_k=0$. We then construct a continuous function $\theta$ satisfying $\theta(0)=0$, $\theta(1/k)=a_k$ for all $k$ and $\theta=a_1$ in $[1, \infty)$.

Now it is clear that $\theta\circ u$ is an axisymmetric solution of (\ref{mcf1}) with initial data $\theta\circ u_0$, again due to Theorem \ref{thm trans invariant}. By our construction of $\theta$, we easily see that $\theta\circ u_0\geq \tilde{u}_0$. Applying Theorem \ref{comparison theorem} for all $T>0$, we get $\theta\circ u\geq \tilde{u}$. This means that $\Gamma_t\subset \tilde{\Gamma}_t$ for any $t\geq 0$. Indeed, for any $p\in \Gamma_t$, we have $u(p, t)=0$, which implies that $\theta\circ u(p, t)=0$ and therefore $\tilde{u}(p, t)=0$. 

We conclude the proof by similarly showing the inclusion $\tilde{\Gamma}_t\subset \Gamma_t$ for any $t\geq 0$.

\end{proof}

\subsection{Finite time extinction}\label{sec extinction}
We give a simple geometric property of the mean curvature flow. The following result shows that an axisymmetric compact surface evolving by its mean curvature shrinks and disappears in finite time.
\begin{thm}[Finite time extinction for bounded evolution]\label{thm extinction}
Suppose that $\{\Gamma_t\}_{t\geq 0}$ denotes an axisymmetric surface evolution of the mean curvature flow. If $\Gamma_0\subset B_r$, for $r>0$, then $\Gamma_t=\emptyset$ when $t>r^2/\sqrt{12}$. 
\end{thm}
\begin{proof}
We may take an axisymmetric $u_0\in C(\mathcal{H})$ with a constant value $C>0$ outside $B_r$ satisfying (\ref{initial surface}) and 
\[
u_0\geq \min\{|p|^4-r^4, C\}.
\]
Taking 
\[
w(p, t)=(p_1^2+p_2^2)^2+12 t(p_1^2+p_2^2)+16p_3^2+12t^2
\]
as in (\ref{explicit solution}),
we easily see that $w^C(p, t):=\min\{w(p, t)-r^4, C\}$ is a solution of (\ref{mcf1}) with initial data $w^C(p, 0)=\min\{|p|^4-r^4, C\}$, by Theorem \ref{thm trans invariant} with $\theta(x)=\min\{x, C\}$. We are therefore led to $u\geq w-r^4$ by Theorem \ref{comparison theorem}. 

It is clear that $w^C(p, t)>0$ when $t>r^2/\sqrt{12}$ for all $p\in \mathcal{H}$, which implies that $u>0$ when $t>r^2/\sqrt{12}$. Hence $\Gamma_t$ defined in (\ref{surface evolution}) is empty when $t>r^2/\sqrt{12}$. Note that the conclusion does not depend on the particular choice of $u_0$, as explained in Corollary \ref{cor independence}.
\end{proof}

\begin{rmk}
Theorem \ref{thm extinction} indicates that a bounded axisymmetric mean curvature flow encounters singularities at a certain time $T>0$.
\end{rmk}

\begin{rmk}
The following result stronger than Theorem \ref{thm extinction} holds: For any continuous solution $u$ of (MCF) with zero level set $\Gamma_t$ for any $t\geq 0$, if $\Gamma_0\subset B_r$ with some $r>0$, then $\Gamma_t=\emptyset$ when $t>r^2/\sqrt{12}$. Here we do not need to assume the axial symmetry of $\Gamma_0$ but we must specify the solution $u$ since it is not known in general whether or not $\Gamma_t$ depends on the choice of $u_0$. 
\end{rmk}

\begin{defi}
We say $T\geq 0$ is the \textit{extinction time} of the mean curvature flow $\Gamma_t$ in the Heisenberg group, if $\Gamma_t\neq\emptyset$ when $t\leq T$ and $\Gamma_t=\emptyset$ when $t>T$.
\end{defi}

We next proceed to investigate the asymptotic profile after normalization for a sphere in the Heisenberg group. It is well-known that in the Euclidean space any normalized compact convex surface converges to a sphere as $t$ tends to the extinction time \cite{H}.  However, the normalized curvature flow from a sphere of radius $r$ in the Heisenberg group looks like an ellipsoid 
\begin{equation}\label{ellipsoid}
E_T:=\{P\in \mathcal{H}: 12T(P_1^2+P_2^2)+16P_3^2=1\} 
\end{equation}
at the extinction time $T=r^2/\sqrt{12}$.

\begin{prop}
Suppose that $\Gamma_t\subset \mathcal{H}$ ($t\geq 0$) is the horizontal mean curvature flow as defined in (\ref{surface evolution}) with $\Gamma_0=\{p\in \mathcal{H}: |p|=r\}$, where $r>0$ is a given radius. Then the extinction time $T=r^2/\sqrt{12}$ and the normalized flow $\Gamma_t/\sqrt{r^4-12t^2}\to E_T$ as $t\to T$, where $E_T$ is given in (\ref{ellipsoid}).
\end{prop}

\begin{proof}
We take
\[
w^C(p, t)=\min\{(p_1^2+p_2^2)^2+12 t(p_1^2+p_2^2)+16p_3^2+12t^2-r^4, C\}
\]
with $C>0$. It is easily seen that $w^C(p, 0)=0$ if and only if $p\in \Gamma_0$. We have also shown that $w^C$ is a solution of (\ref{mcf1}). We track the evolution by setting $\Gamma_t=\{p\in \mathcal{H}: w^C(p, t)=0\}$ for all $t\geq 0$. It is clear that $\Gamma_t=\emptyset$ when $t>r^2/\sqrt{12}$ and $\Gamma_t\neq \emptyset$ when $t\leq r^2/\sqrt{12}$. 
For any $p(t)=(p_1(t), p_2(t), p_3(t))\in \Gamma_t$, we have
\begin{equation}\label{level set estimate}
12t(p_1^2(t)+p_2^2(t))+16p_3^2(t)\leq r^4-12t^2.
\end{equation}
We normalize the flow by letting $P(t)=p(t)/\sqrt{r^4-12t^2}$ for any $p(t)\in \Gamma_t$. Then (\ref{level set estimate}) is written as
\begin{equation}\label{level set estimate2}
12t(P_1^2(t)+P_2^2(t))+16P_3^2(t)\leq 1.
\end{equation}
By setting 
\[
U(P, t)=(P_1^2+P_2^2)^2(r^4-12t^2)+12t(P_1^2+P_2^2)+16P_3^2-1
\]
we get
\[
0=\frac{1}{r^4-12t^2}w^C(\sqrt{r^4-12t^2}P(t), t)=U(P(t), t)).
\]
Sending the limit as $t\to T$ with (\ref{level set estimate2}) taken into account, we obtain
\[
12 T(P_1^2(T)+P_2^2(T))+16P_3^2(T)= 1
\]
for the limit $P(T)$ of any subsequence of $P(t)$ as $t\to \infty$. 
The consequence above amounts to saying that the limit of the set $\Gamma_t/\sqrt{r^4-12t^2}$ is contained in $E_T$.

On the other hand, for any $P=(P_1, P_2, P_3)\in E_T$, we have 
\[
w^C(\sqrt{r^4-12t^2}\lambda P)=(r^4-12t^2)W(\lambda, P, t),
\]
where $\lambda>0$ and
\[
W(\lambda, P, t)=(\lambda^4(P_1^2+P_2^2)^2(r^4-12t^2)+\lambda^2-1+12\lambda^2(t-T)(P_1^2+P_2^2).
\]
One may take $\lambda(t)>0$ such that $w^C(\sqrt{r^4-12t^2}\lambda(t)P, t)=0$; in other words, $\lambda(t)P\in \Gamma_t/\sqrt{r^4-12t^2}$. Moreover, $\lambda(t)\to 1$ as $t\to T$, which implies that $P$ belongs to the limit of a sequence of elements in $\Gamma_t/\sqrt{r^4-12t^2}$.

In conclusion, we obtain $\Gamma_t/\sqrt{r^4-12t^2}\to E_T$ as $t\to T$. 
\end{proof}

We stress that this result is very different from that in the Euclidean space. The normalized asymptotic shape of horizontal mean curvature flow in the Heisenberg group starting from a ball is an ellipsoid. Moreover, the shape of the ellipsoid depends on the extinction time $T$ and therefore the size of the initial surface. It would be interesting to show this result for a general compact and convex initial surface.

\end{document}